\documentclass[10pt]{amsart}
\usepackage[latin1]{inputenc}
\usepackage[T1]{fontenc}
\usepackage{epsfig}
\usepackage{mathrsfs}
\usepackage{color}
\usepackage[english]{babel}
\usepackage{amsthm}
\usepackage{amsmath}
\usepackage{amsfonts}
\usepackage{amssymb}
\usepackage{graphicx}
\usepackage[colorlinks=true,hyperindex=true,bookmarks=true]{hyperref}
\hypersetup{
    colorlinks,
    citecolor=blue,
    filecolor=purple,
    linkcolor=red,
    urlcolor=blue
}

\def\rconj{w}
\def\rconjmoins{v}

\newtheorem{lem}{Lemma}[section]
\newtheorem{prop}{Proposition}[section]
\newtheorem{cor}{Corollary}[section]
\newtheorem{thm}{Theorem}[section]
\newtheorem{defi}{Definition}[section]
\newtheorem{rem}{Remarks}[section]
\newtheorem{remark}{Remark}[section]
\newfont{\sBlackboard}{msbm10 scaled 900}

\newcommand{\dd}     {{\rm d}}
\newcommand{\mylabel}[1]{\label{#1}
            \ifx\undefined\stillediting
            \else \fbox{$#1$}\fi }
\newcommand{\BE}{\begin{equation}}

\newcommand{\EEQ}{\end{equation}}
\newcommand{\rfb}[1]{\mbox{\rm
   (\ref{#1})}\ifx\undefined\stillediting\else:\fbox{$#1$}\fi}

\newfont{\Blackboard}{msbm10 scaled 1200}

\newfont{\roma}{cmr10 scaled 1200}

\def\CC{\rm \hbox{C\kern-.56em\raise.4ex
         \hbox{$\scriptscriptstyle |$}\kern+0.5 em }}

\newcommand{\rline}  {{\mathbb R}}

\newcommand{\ud}{\mathrm{d}}
\newcommand{\re}{\mathrm{Re}}
\newcommand{\R}{\mathbb{R}}

\newcommand{\N}{\mathbb{N}}

\newcommand{\mm}    {{\hbox{\hskip 0.5pt}}}

\newcommand{\bluff} {{\hbox{\raise 15pt \hbox{\mm}}}}
\makeatletter
\def\section{\@startsection {section}{1}{\z@}{-3.5ex plus -1ex minus
    -.2ex}{2.3ex plus .2ex}{\large\bf}}
\makeatother
\def\be{\begin{equation}}
\def\ee{\end{equation}}
\def\beqs{\begin{eqnarray*}}
\def\eeqs{\end{eqnarray*}}
\def\pd{\partial}
\def\ds{\displaystyle}
\let \div \relax
\DeclareMathOperator{\div}{div}

\newcommand{\nablag}{\nabla_{\!\! g}}
\newcommand{\nablatg}{\nabla_{\!\! \tilde g}}
\newcommand{\Opt}{\ensuremath{\mathrm{Op}_{\mbox{\tiny ${\mathsf T}$}}}}

\newcommand{\Ssct}{S_{\mbox{\tiny ${\mathsf T}$},\tau}}
\let \Re \relax
\DeclareMathOperator{\Re}{Re}
\let \Im \relax
\DeclareMathOperator{\Im}{Im}
\DeclareMathOperator{\supp}{supp}
\newcommand{\Lsct}{\Lambda_{\mbox{\tiny ${\mathsf T}$},\tau}}
\newcommand{\eps}{{\varepsilon}}
\renewcommand{\SS}{\ensuremath{\mathscr S}}
\newcommand{\Con}{\ensuremath{\mathscr C}}

\begin{document}
\thispagestyle{empty}
\title[Logarithmic stabilization of an acoustic system]{Logarithmic stabilization of an acoustic system with a damping term of Brinkman type}

\author{Ka\"{\i}s AMMARI}
\address{UR Analysis and Control of PDEs, UR 13ES64, Department of Mathematics, Faculty of Sciences of Monastir, University of Monastir, 5019 Monastir, Tunisia and LMV/UVSQ/Paris-Saclay, France} \email{kais.ammari@fsm.rnu.tn}

 \author{Fathi Hassine}
\address{UR Analysis and Control of PDEs, UR 13ES64, Department of Mathematics, Faculty of Sciences of Monastir, University of Monastir, 5019 Monastir, Tunisia} \email{fathi.hassine@fsm.rnu.tn}

\author{Luc ROBBIANO}
\address{Université Paris-Saclay, UVSQ, CNRS, Laboratoire de Mathématiques de Versailles, 78000, Versailles, France.
}
\email{luc.robbiano@uvsq.fr}
\date{}

\begin{abstract}
We study the problem of stabilization for the acoustic system with a spatially distributed damping. Without imposing any hypotheses on the structural properties of the damping term, we identify logarithmic decay of solutions with growing time. Logarithmic decay rate is shown by using a frequency domain method and combines a contradiction argument with the multiplier technique and a new Carleman estimate to carry out a special analysis for the resolvent.
\end{abstract}


\maketitle

\tableofcontents

\section{Introduction} \label{intro}
We consider the following system of equations:
\be
\label{fluide}
\left\{
\begin{array}{l}
u_t + \nabla r + b \, u = 0, \,\hbox{ in } \Omega \times \rline^+, \\
r_t + \div u = 0, \,\hbox{ in } \Omega \times \rline^+, \\
u \cdot n = 0, \,\hbox{ on } \Gamma \times \rline^+, \\
u(0,x) = u^0(x), \, r(0,x) = r^0(x), \, x \in \Omega,
\end{array}
\right.
\ee
where $\Omega$ is a bounded domain in $\rline^d, \, d\geq2$,
with a smooth boundary $\Gamma$, $\div=\nabla\cdot$ is the divergence operator and $b\in L^{\infty}(\Omega)$,
with $b \ge 0$ on $\Omega$ and {{such that}}
\begin{equation}\label{eq:a}
\exists\ b_->0 \mbox{ such that } b \ge b_-
\mbox{ on } \omega.
\end{equation}
Here $\omega \ne \emptyset$ stands for the open {{subset}} of $\Omega$ on which the feedback is active.
As usual $n$ denotes the unit outward normal vector along $\Gamma$.

\medskip

The system of equations (\ref{fluide}) is a linearization of the \emph{acoustic equation} governing the propagation of 
acoustic waves in a compressible medium, see Lighthill \cite{Lighthill:78,Lighthill:52,Lighthill:54}, where $b \, u$ represents a damping term of Brinkman type. 
This kind of damping arises also in the process of homogenization (see Allaire \cite{Allaire:91}), and is frequently used as a suitable \emph{penalization} 
in fluid mechanics models, see Angot, Bruneau, and Fabrie \cite{AngotBruneauFabrie:99}. Our main goal is to prove the logarithmic decay of solutions of (\ref{fluide}) with growing time.  
 
\medskip

Let $L^2(\Omega)$ denote the standard Hilbert space of square integrable functions in $\Omega$
and its closed subspace $L^2_m(\Omega)=\{f\in L^2(\Omega): \int_\Omega f(x)\,dx=0\}$.
To avoid abuse of notation,
we shall write $\|\cdot\|$ for the $L^2(\Omega)$-norm or the $L^2(\Omega)^d$-norm.

\medskip

Denoting  $H = (L^2(\Omega))^d \times L^2_m(\Omega)$, we introduce the operator
$$
{\mathcal A} = \left(
\begin{array}{ll}
0  & \nabla \\
\div  & 0
\end{array}
\right) : {\mathcal D}({\mathcal A}) = \left\{(u,r) \in H, \, (\nabla r,\div u) \in H, \, u \cdot n_{|\Gamma} = 0 \right\} \subset  H \rightarrow H,
$$
and
$$
{\mathcal B} =  \left( \begin{array}{ll} \sqrt{b} \\ 0 \end{array} \right) \in {\mathcal L}((L^2(\Omega))^d,H), \,
{\mathcal B}^* = \left( \begin{array}{cc} \sqrt{b} & 0 \end{array} \right) \in {\mathcal L}(H, (L^2(\Omega))^d).
$$
We recall that for $u\in  (L^2(\Omega))^d$ with $\div u\in L^2(\Omega)$, $u \cdot n_{|\Gamma} $ make sens in $H^{-1/2}(\Gamma)$
(see Girault-Raviart~\cite[Chp 1, Theorem 2.5]{Girault-Raviart}).
\medskip

Accordingly, the problem (\ref{fluide}) can be recasted in an abstract form:
\be
\label{cauchy}
\left\{
\begin{array}{l}
Z_t (t) + {\mathcal A} Z(t) + {\mathcal B} {\mathcal B}^* Z(t) = 0, \, t > 0, \\
Z(0) = Z^0,
\end{array}
\right.
\ee
where $Z = (u, r)$,
or, equivalently, 
\be
\label{cauchybis}
\left\{
\begin{array}{l}
Z_t (t) = {\mathcal A}_d Z(t), \, t > 0, \\
Z(0) = Z^0,
\end{array}
\right.
\ee
with ${\mathcal A}_d = - {\mathcal A} - {\mathcal B}{\mathcal B}^*$ with ${\mathcal D}({\mathcal A}_d) = {\mathcal D}({\mathcal A}).$

\medskip

It can be shown (see \cite{AFN}) that for any initial data $(u^0, r^0) \in {\mathcal D}({\mathcal A})$ the 
problem \eqref{fluide} admits a unique solution
$$(u,r) \in C([0,\infty); {\mathcal D}({\mathcal A})) \cap C^1([0, \infty); H).$$ Moreover, the solution $(u,r)$ satisfies,
the energy identity
\be
\label{energyid}
E(0) - E(t) =
 \int_0^t
\left\|\sqrt{b} \, u(s)\right\|_{(L^2(\Omega))^d}^2 \dd s, \ \mbox{for all}\ t \geq 0
\ee
with
\be
\label{energy}
E(t) = \frac{1}{2} \,\left\|(u(t),r(t)) \right\|^2_{H}, \, \forall \, t \geq 0, 
\ee
where we have denoted
$$
\left\langle (u,r), (v,p)\right\rangle_H = \int_\Omega \left(u(x) . v(x) + r(x) p(x) \right) \, \ud x , \,  \left\| (u,r)\right\|_H = \sqrt{\int_\Omega \left(\left|u(x)\right|^2 + r^2(x) \right) \, \ud x}.
$$

\medskip

Using (\ref{energyid}) and a standard density argument, we can extend the solution operator for  data 
$(u^0, r^0) \in H$. Consequently, we associate with the problem (\ref{fluide}) (or to the abstract Cauchy problems (\ref{cauchy}) or
(\ref{cauchybis})) a  semi-group that is globally bounded in $H$. 

\medskip

As the energy $E$ is nonincreasing along trajectories, we want to determine the set of initial data $(u^0, r^0)$ for which 
\begin{equation}\label{stab}
E(t) \to 0 \ \mbox{as} \ t \to \infty.
\end{equation}
Such a question is of course intimately related to the structural properties of the function $b$, notably to the geometry of the set $\omega$ on which the damping is effective. In fact, when the damping term is globally distributed Ammari, Feireisl and Nicaise \cite{AFN} showed an exponential decay rate of the energy by the means of an observability inequality associated with the conservative problem of \eqref{fluide}. Besides, it is also shown that if the damping coefficient is not uniformly positive definite (i.e $\ds\inf_{x\in\Omega}b(x)=0$) then the system \eqref{fluide} is not exponentially stable. In this paper we consider a damping which is locally distributed over the domain $\Omega$ without any geometrical control condition in particular this including the case when the damping coefficient is not uniformly positive defined. So we expect to prove a weaker decay rate then given in \cite{AFN}. More precisely, we prove a logarithm decay rate of the energy. Our approach is based in the frequency domain method which consist to prove an exponential loss on the resolvent estimate \cite{batty, burq, Duyckaerts} where the main tool for establishing a such decay is the Carleman estimate. 

\medskip

The theory of Carleman estimates for scalar equations is rather well developed by now. We refer to H\"ormander \cite{hormander} and Lebeau and Robbiano \cite{lebeau, lebeau-robbiano1, lebeau-robbiano2} for the second-order elliptic and hyperbolic PDE's and to Isakov \cite{isakov} second-order parabolic and Schr\"odinger operators. However, it turned out that Carleman estimates for systems in more than two variables are difficult to obtain and still somehow very limited: The first results to systems go back to then Carleman's original work \cite{carleman} which is written for a system in two independent variables, and we refer to Calder\'on \cite{calderon} and Kreiss \cite{kreiss} for more relevant systems. Recently, Eller and Toundykov \cite{eller-toundykov} have established a Carleman estimate for some first-order elliptic systems. This estimate is extended to elliptic boundary value problems provided the boundary condition satisfies a Lopatinskii-type requirement. In this paper we provide a Carleman estimate for a system of first-order which does not fit into the same framework as that of Eller and Toundykov \cite{eller-toundykov}. Unlike their approach, our method is based into the H\"ormander approach which is essentially based on the sub-ellipticity condition and the G{\aa}rding inequality in order to control the non-elliptical regions.

The paper is organized as follows. Section \ref{section2} summarizes some well known facts concerning the acoustic system \rfb{fluide}. In section \ref{section3}, we establish a new Carleman estimate needed for the stabilization problem of the system (\ref{fluide}).  
In Section \ref{section4}, we prove the logarithmic stability for the system \rfb{fluide}.

\section{Preliminaries}
\label{section2}

We start with a simple observation that the problem (\ref{fluide}) can be viewed as a bounded (in $H$) perturbation of the conservative system 
\be
\label{fluidec}
\left\{
\begin{array}{l}
u_t + \nabla r  = 0, \,\hbox{ in } \Omega \times \rline^+, \\
r_t + \div u = 0, \,\hbox{ in } \Omega \times \rline^+, 
\end{array}
\right.
\ee
which can be recast as the standard \emph{wave equation}
\[
r_{tt} - \Delta r = 0.
\]
Consequently, the basic existence theory for (\ref{fluide}) derives from that of (\ref{fluidec}).
Hence ${\mathcal A}_d$ generates a $C_0$-semigroup $(S(t))_{t\geq 0}$ in $H$ that is even of contraction because ${\mathcal A}_d$ is dissipative (see \rfb{energyid}).

The first main difficulty is that
the operator ${\mathcal A}_d$ possesses a non-trivial (and large) kernel that is left invariant by the evolution.
Indeed if  $(u,r)$ belongs to $\ker {\mathcal A}_d$, then it is solution of the ``stationary'' problem 
\begin{equation}\label{st1}
\nabla r + b u = 0, \ \div u = 0, \,\hbox{ in } \Omega.
\end{equation} 
Thus  multiplying the first identity of  (\ref{st1}) by $\overline{u}$ and integrating over $\Omega$ yields 
\[
\int_\Omega (\nabla r \cdot \overline{u}+ b |u|^2)\,\ud x=0.
\]
By an integration by parts, using the fact that $u$ is solenoidal and the 
boundary condition $u \cdot n= 0$ on  $\Gamma$, we get
\[
\int_\Omega \nabla r \cdot \overline{u}\,\ud x=0,\]
and therefore we obtain 
\[
\int_\Omega  b |u|^2\,\ud x=0.
\]
In other words, we have
\[
 u = 0 \ \mbox{on}\ \supp b ,
\]
and coming back to (\ref{st1}), we find
\[
\nabla r = 0.
\]
Accordingly, we have shown that
\[
\ker {\mathcal A}_d = \{(u, 0)\in {\mathcal D}({\mathcal A}) \ | \ \div u = 0, \ u|_{ \supp b} = 0, 
\ u \cdot n|_{\Gamma} = 0 \}.
\]
For shortness set $E=\ker {\mathcal A}_d$ and
introduce also   its orthogonal complement $H_0$ in $H$. 

It is easy to check that 
\[
\left< {\mathcal A}_d ( w, s ), ( u, r) \right>_H = 0 \ \mbox{for any}\ 
(w, s) \in {\mathcal D}({\mathcal A}), \ (u,r) \in E;
\] 
in particular, the semi-group  associated with (\ref{fluide}) leaves both $E$ and $H_0$ invariant. 
Consequently, the decay property (\ref{stab}) may only hold  for   initial data emenating from the set $H_0$.

\medskip

The following observation can be shown by a simple density argument:

\begin{lem}
The solution $(u,r)$ of \rfb{fluide} with
initial datum in ${\mathcal D}(\mathcal{A}_d)$ satisfies
\be
\label{deriveeenergy}
E'(t) = - \int_\Omega b\left|u \right|^2dx\leq0.
\ee
Therefore the energy is non-increasing and \rfb{energyid} holds for all initial datum in $H$.
\end{lem}

As already shown in the above, the strong stability result (\ref{stab}) may hold only if 
we take the initial data
\[
(u^0, r^0) \in H_0 = \ker [{\mathcal A}_d]^\perp. 
\]
There are several ways how to show (\ref{stab}), here
we make use of the following result
due to Arendt and Batty \cite{arendt:88}:

\begin{thm}\label{thmArendtBatty}
Let $(T(t))_{t\geq0}$ be a bounded $C_0$-semigroup on a reflexive Banach space $X$. Denote by $A$ the generator of $(T(t))$ and by $\sigma(A)$ the spectrum of $A$. If $\sigma(A)\cap i\mathbb{R}$ is countable and no eigenvalue of $A$ lies on the imaginary axis, then $\ds \lim_{t\rightarrow+\infty} T(t)x = 0$ for
all $x\in X$.
\end{thm}

In view of this theorem we need to identify the spectrum of
${\mathcal A}_d$ lying on the imaginary axis, and we have according to \cite{AFN}:

\begin{itemize}
\item
Suppose that $|\omega| > 0$.
If $\lambda$ is a non-zero real number, then $i\lambda$  is not an eigenvalue of ${\mathcal A}_d$.
\item
Suppose that $|\omega| > 0$.
If $\lambda$ is a non-zero real number, then $i\lambda$ belongs to the resolvent set $\rho({\mathcal A}_d)$ of ${\mathcal A}_d$.
\end{itemize}

Now, Theorem \ref{thmArendtBatty} leads to
\begin{cor}[\cite{AFN}]
\label{cconv}
Let $(u, r)$ be the unique semi-group solution of the problem (\ref{fluide}) emanating from the initial data 
$(u^0,r^0) \in H$. Let $P_E$ be the orthogonal projection onto the space $E = \ker [{\mathcal A}_d ]$ in $H$, and let 
\[
(w,s) = P_E (u^0, r^0).
\] 
Then
\[ 
\| (u,r)(t, \cdot) - (w, s) \|_{H} \to 0 \ \mbox{as}\ t \to \infty
\]
\end{cor}

We now state the main result of this article. We begin by a proposition on  an estimate of the resolvent.
 %
 %
 \begin{prop}
 	\label{prop: resolv exp}
There exist $C>0$ such that for every  $|\mu|\ge 1$, and  $(f,g)\in H=  (L^2(\Omega))^d \times L^2_m(\Omega)$, 
the solution $(u,r)\in {\mathcal D}({\mathcal A})$ of $({\mathcal A}_d +i\mu )(u,r)=(f,g)$  satisfied   
\begin{equation}
\| (u,r) \|_H\le C e^{C|\mu|}\| (f,g)  \|_H,
\end{equation}
or equivalently
\begin{equation}
\| ({\mathcal A}_d +i\mu )^{-1}\|_{{\mathcal L}(H)}\le C e^{C|\mu|}.
\end{equation}
 \end{prop}  
We recall the following result.
 %
 %
\begin{thm}
	\label{thm: resolvent decay}
Let $B$ a generator of a $C_0$-semigroup $(T(t))_{t\geq 0}$ on $\mathcal H$, a Hilbert space.
We assume
\begin{align}
& \|T(t)\|_{{\mathcal L}(H)} \text{ is uniformly bounded with respect } t\ge0,\\
& B+i\mu \text{ is invertible for every } \mu\in \R,\\
&\text{There exists } C>0 \text{ such that } \| (B +i\mu )^{-1}\|_{{\mathcal L}(H)}\le C e^{C|\mu|}.
\end{align}
Then there exist $C_1>0$ such that for all $u\in { \mathcal D} (B)$ we have 
\[
\| T(t) u \|_{\mathcal H} \le  C_1\frac{\| Bu\|_{\mathcal H}}{\ln(3+t)}, \, \forall \, t \geq 0.
\]
\end{thm}
One has also, for every $k\ge 1$ there exists $ C_2>0$ such that if $u\in { \mathcal D} (B^k)$,  we have 
\[
\| T(t) u \|_{\mathcal H} \le  C_1\frac{\| B^ku\|_H}{\ln^k(3+t)}, \, \forall \, t \geq 0.
\]
A weak version of this theorem was first proven by Lebeau~\cite{lebeau}, next Burq~\cite{burq} gives the precise statement. 
We also refer to Batty and Duyckaerts \cite{batty} for some generalizations.

On $ H_0= \ker [{\mathcal A}_d]^\perp$, as seen above ${\mathcal A}_d +i\mu $ is invertible on $H_0$, in fact
${\mathcal A}_d$ is invertible on $H_0$ and ${\mathcal A}_d +i\mu $ is invertible on $H$ for $\mu\ne 0$. 
The semigroup is bounded as the norm on $H$ is non-increasing by~\eqref{energyid}.
With  Proposition~\ref{prop: resolv exp}, we can apply Theorem~\ref{thm: resolvent decay}. We then obtain.
 %
 %
\begin{thm}
Let $(u, r)$ be the unique semi-group solution of the problem (\ref{fluide}) emanating from the initial data 
$(u^0,r^0) \in {\mathcal D}({\mathcal A})$. Let $P_E$ be the orthogonal projection onto the space 
$E = \ker [{\mathcal A}_d ]$ in $H$, and let 
\[
(w,s) = P_E (u^0, r^0).
\] 
Then
\[ 
\| (u,r)(t, \cdot) - (w, s) \|_{H}\le   C \frac{\| {\mathcal A}_d (u^0, r^0) \|_{H}}{\ln(3+t)}, \, \forall \, t \geq 0,
\]
for some $C>0$ independent of $ (u^0, r^0)$.
\end{thm}
Proposition~\ref{prop: resolv exp} is obtained from Carleman estimates. We need two kinds of such estimates, first an estimate
far away the boundary, second an estimate up to the boundary. Both estimates are proven in the next section.


\section{Carleman estimates}\label{section3}
\setcounter{equation}{0}
Let $\Omega$ be an open bounded subset of $\R^{d}$. Let $(u,r)$ be a solution of the resolvent problem 
$({\mathcal A}_0 +i\mu )(u,r)=(f,g)\in   (L^2(\Omega))^d \times L^2_m(\Omega)$, that is
\begin{equation}
\begin{cases}
	\label{eq: reslov bd}
&-\nabla r+ i\mu u=f \text{ in } \Omega, \\
&-\div u+i\mu r=g  \text{ in }  \Omega.
\end{cases}
\end{equation}
Here we moreover assume that $(u,r)$ are supported in $K\subset \Omega$ where $K$ is a compact set.
Taking the divergence of the first line and using that $\div u= i\mu r-g $, we obtain
\begin{equation}
	\label{ASCE2}
-\Delta r-\mu^2 r=i\mu g +\div f  \text{ in } \Omega .
\end{equation}
We have to give a Carleman estimate for the solution of this type of equation. This is done in Section~\ref{subsec: away bd}. 
to do that we need some tools on pseudo-differential operators we recall below.
\subsection{Pseudo-differential operators}
	\label{sec: Pseudo-differential operators}
We start this section with some useful notations. If $\alpha=(\alpha_{1},\dots,\alpha_{n})\in\mathbb{N}^{n}$ 
is a multi-index, we introduce the following notation:
\[
\xi^{\alpha}=\xi_{1}^{\alpha_{1}}\ldots\xi_{n}^{\alpha_{n}},\;\;\partial^{\alpha}=\partial_{x_{1}}^{\alpha_{1}}
\ldots\partial_{x_{n}}^{\alpha_{n}},\;\;D^{\alpha}=D_{1}^{\alpha_{1}}\ldots D_{n}^{\alpha_{n}}\;\textrm{ and }\; 
|\alpha|=\alpha_{1}+\dots+\alpha_{n}
\]
where $\displaystyle D_{k}=-i\frac{\pd}{\pd x_{k}}=-i\pd_{x_k}$. We denote by $\mathscr{C}_{c}^{\infty}(V)$ the set of 
functions of class $\mathscr{C}^{\infty}$ compactly supported in $V$. For a compact subset $K$ of $\mathbb{R}^{n}$, 
we note by $\mathscr{C}_{c}^{\infty}(K)$ the set of functions in $\mathscr{C}_{c}^{\infty}(\mathbb{R}^{n})$  supported in $K$. 
The space $L^{2}(V)$ is equipped with the usual norm denoted by $\|u\|_{0}$. For $s\in \mathbb{N}$ we set 
$H^{s}(V)=\{u\in \mathscr{D}'(V);\,\partial^{\alpha}u\in L^{2}(V)\;\forall\,|\alpha|\leq s\}$. The Schwartz space 
$\mathscr{S}(\mathbb{R}^{n})$ is the set of functions of $\mathscr{C}^{\infty}$ class with rapid decay rate. 
Its dual, $\mathscr{S}'(\mathbb{R}^{n})$ is the set of temperate  distributions. If $u\in \mathscr{S}(\mathbb{R}^{n})$ its 
Fourier transform denoted by $\hat{u}$ is defined by $\ds\hat{u}(\xi)=\int_{\R^{n}}e^{-iy.\xi}u(y)\,\ud y$ where 
$\ds y.\xi=\sum_{i=1}^{n}y_{i}\xi_{i}$ stands for the euclidean inner production in $\R^{n}$. Let $f$ and $g$ be two 
smooth functions defined in $V\times\mathbb{R}^{n}$, we define the Poisson bracket by 
$\displaystyle \{f,g\}=\sum_{j=1}^{n}(\partial_{\xi_{j}}f.\partial_{x_{j}}g-\partial_{x_{j}}f.\partial_{\xi_{j}}g)$. 
And if $A$ and $B$ are two operators we define there commutator by $[A,B]=A\circ B-B\circ A$.
\begin{defi}
	\label{def: pseudo}
Let $a(\,.\, ,\,.\, ,\tau)\in \mathscr{C}^{\infty}(\mathbb{R}^n\times\mathbb{R}^{n})$ where $\tau\geq 1$ is a 
large parameter, such that for every muti-index $\alpha,\beta\in\mathbb{N}^{n}$ we have
\[
|\partial_{x}^{\alpha}\partial_{\xi}^{\beta} a(x,\xi,\tau)| \leq C_{\alpha,\beta}\left\langle \xi,\tau\right\rangle^{m-|\beta|}, \quad \forall\,x\in\mathbb{R}^{n},\;\forall\,\xi\in\mathbb{R}^{n},\;\forall\,\tau\geq 1,
\]
where we denoted by $\left\langle \xi,\tau\right\rangle=(|\xi|^2+\tau^{2})^{\frac{1}{2}}$. In this case we say that $a$ is a symbol of order $m$ and we write $a\in S_{\tau}^{m}$. We call principal symbol of $a\in S_{\tau}^{m}$ the equivalence class of $a$ in $S_{\tau}^{m}/S_{\tau}^{m-1}$. We also define $\ds S_{\tau}^{-\infty}=\bigcap_{r\in\R}S_{\tau}^{r}$ and $\ds S_{\tau}^{+\infty}=\bigcup_{r\in\R}S_{\tau}^{r}$.
\end{defi}
\begin{defi}
We define the pseudo-differential operator of order $m$ by 
\begin{equation*}
a(x,D,\tau)u(x)=Op(a)u(x)=\frac{1}{(2\pi)^{n}}\int_{\mathbb{R}^{n}} e^{i x.\xi}a(x,\xi,\tau)\hat{u}(\xi)\,\ud \xi,\quad  \forall\,u\in\mathscr{S}(\mathbb{R}^{n}),
\end{equation*}
where $a\in S_{\tau}^{m}$. The set of the pseudo-differential operator of order $m$ is denoted by $\Psi_{\tau}^{m}$. If $A\in \Psi_{\tau}^{m}$, we denote by $\sigma_{p}(A)$ his principal symbol.
\end{defi}
\begin{rem}
Let $s\in\mathbb{R}$ for $u\in\mathscr{S}'(\mathbb{R}^{n})$ we set the following norm
$$
\|u\|_{\tau,s}=\|\Lambda_{\tau}^{s}u\|_{0}\;\textrm{ with }\Lambda_{\tau}^{s}:=Op(\langle\xi,\tau\rangle^{s}).
$$
Hence we can define the corresponding space
$$
H_{\tau}^{s}(\mathbb{R}^{n})=\{u\in \mathscr{S}'(\mathbb{R}^{n})\,;\;\|u\|_{\tau,s}<\infty\}.
$$
\end{rem}
\begin{thm}\label{-1}
Let $s\in\mathbb{R}$ and $a(x,\xi,\tau)\in S_{\tau}^{m}$, then the operator $Op(a):H_{\tau}^{s}\longrightarrow H_{\tau}^{s-m}$ maps continuously and uniformly for $\tau>1$.
\end{thm}
\begin{lem}
Let $m\in \mathbb{R}$ and $a_j\in S_{\tau}^{m-j}$ with $j \in \mathbb{N}$. Then there exist $a\in S_{\tau}^{m}$ such that
\[
\forall\,N\in \mathbb{N},\quad a-\sum_{j=0}^{N} a_j \in S_{\tau}^{m-N-1}.
\]
We then write $\ds a\sim \sum_{j}a_j$. The symbol $a$ is unique up to $S_{\tau}^{-\infty}$ in the sens that the difference of two symbols is in $S_{\tau}^{-M}$ for all $M \in\mathbb{N}$. Hence, we identify $a_{0}$ with the principal symbol of $a$.
\end{lem}
\begin{thm}\label{9}
Let $a\in S_{\tau}^{m}$ and $b\in S_{\tau}^{m'}$, then $Op(a)\circ Op(b)=Op(c)$ with $c\in S_{\tau}^{m+m'}$ which admits the following asymptotic expansion
\[
c(x,\xi,\tau)\sim\sum_{\alpha}\frac{1}{i^{|\alpha|}\alpha!}\partial_{\xi}^{\alpha}a(x,\xi,\tau)\partial_{x}^{\alpha}b(x,\xi,\tau).
\]
\end{thm}
\begin{thm}
Let $a\in S_{\tau}^{m}$ and $b\in S_{\tau}^{m'}$, then $[Op(a),Op(b)]=Op(c)$ with $c\in S_{\tau}^{m+m'-1}$ and principal symbol
$$
\sigma(c)(x,\xi,\tau)=\frac{1}{i}\{a,b\}(x,\xi,\tau)
$$
which admits the following asymptotic expansion
\[
c(x,\xi,\tau)\sim\sum_{\alpha}\frac{1}{i^{|\alpha|}\alpha!}\left(\partial_{\xi}^{\alpha}a(x,\xi,\tau)\partial_{x}^{\alpha}b(x,\xi,\tau)-\partial_{\xi}^{\alpha}b(x,\xi,\tau)\partial_{x}^{\alpha}a(x,\xi,\tau)\right).
\]
\end{thm}
\begin{thm}\label{-2}
Let $a\in S_{\tau}^{m}$, then $Op(a)^{*}=Op(b)$ with $b\in S_{\tau}^{m}$ which admits the following asymptotic expansion
\[
b(x,\xi,\tau)\sim \sum_{\alpha}\frac{1}{i^{|\alpha|}\alpha !}\partial_{\xi}^{\alpha}\partial_{x}^{\alpha}\bar{a}(x,\xi,\tau).
\]
In particular we have $\sigma_{p}(Op(a)^{*})=\bar{a}$.   
\end{thm}
\begin{thm}[G{\aa}rding inequality]\label{ASC8} 
Let $K$ be a compact subset of $\mathbb{R}^{n}$ and $a(x,\xi,\tau)\in S_{\tau}^{m}$, of principal symbol $a_{m}$. 
We suppose that there exist $C>0$ and $R>0$ such that
\[
\re\,a_{m}(x,\xi,\tau)\geq C\langle\xi,\tau\rangle^{m},\quad\forall\,x\in K,\;\xi\in\mathbb{R}^{n},\;\tau\geq1,\;\langle\xi,\tau\rangle\geq R .
\]
Then for any $0<C'<C$ there exists $\tau_{*}>0$ we have
\[
\re(Op(a)u,u)_{L^{2}(\R^{n})}\geq C'\|u\|_{\tau,\frac{m}{2}}^{2},\quad \forall\,u\in \mathscr{C}_{c}^{\infty}(K),\;\tau\geq\tau_{*}.
\]
\end{thm}
\subsection{Local Carleman estimate away from the boundary}
	\label{subsec: away bd}
We set the operator 
$$
P(x,D)=-\mu^{2}-\Delta,
$$
a real values function $\varphi$ and then we define the conjugate operator by 
$$
P_{\varphi}(x,D)=e^{\tau\varphi}P(x,D)e^{-\tau\varphi},
$$
where $\mu$ is a parameter that depends  on $\tau$, precisely we suppose that 
\begin{equation}
	\label{hyp: mu eq tau}
c_{0}\tau\leq |\mu|\leq c_{0}'\tau \quad\forall\, \tau\geq 1,
\end{equation}
for some constants $c_{0}'>c_{0}>0$. Then we have 
\[
P_{\varphi}(x,D)\rconj=-\mu^{2}\rconj-\Delta \rconj+2\tau\nabla\varphi.\nabla \rconj-\tau^{2}|\nabla\varphi|^{2}\rconj+\tau\Delta\varphi \rconj	
\]
whose symbol is given by
\[
\sigma(P_{\varphi})=|\xi|^{2}+2i\tau\nabla\varphi.\xi-\tau^{2}|\nabla\varphi|^{2}+\tau\Delta\varphi-\mu^{2}
\]
and with principal symbol $p_{\varphi}$ given by
\[
p_{\varphi}(x,\xi,\tau)=|\xi+i\tau\nabla\varphi|^{2}-\mu^{2}=|\xi|^{2}+2i\tau\nabla\varphi.\xi-\tau^{2}|\nabla\varphi|^{2}-\mu^{2}.
\]
We define the following self-adjoint operators
\[
Q_{2}=\frac{P_{\varphi}+P_{\varphi}^{*}}{2}\quad\textrm{and}\quad Q_{1}=\frac{P_{\varphi}-P_{\varphi}^{*}}{2i}
\]
with principal symbols respectively
\[
q_{2}(x,\xi,\tau)=|\xi|^{2}-\tau^{2}|\nabla\varphi|^{2}-\mu^{2}\quad\textrm{and}\quad q_{1}(x,\xi,\tau)=2\tau\nabla\varphi.\xi.
\]
Noting that $P_{\varphi}=Q_{2}+iQ_{1}$ and $p_{\varphi}=q_{2}+iq_{1}$.

We assume that the weight function $\varphi\in\mathscr{C}^{\infty}(\mathbb{R}^{n},\mathbb{R})$ satisfies the following sub-ellipticity condition in $K$ a compact set of $\R^d$, if
\begin{align}
	\label{sub-ellipticity}
	&|\nabla\varphi|>0 \text{ in }  K\\
	& \forall\,(x,\xi,\tau)\in K\times\mathbb{R}^{n}\times[1,+\infty);\;p_{\varphi}(x,\xi,\tau)=0\Longrightarrow \{q_{2},q_{1}\}(x,\xi,\tau)\geq C\langle\xi,\tau\rangle^{3}>0. \notag
\end{align}
Note that the constant $C$ does not depend on $\mu$ assuming \eqref{hyp: mu eq tau}.
\begin{remark}
Noting that 
$$
p_{\varphi}(x,\xi,\tau)=0\;\Longleftrightarrow\;|\xi|^{2}=\tau^{2}|\nabla\varphi|^{2}+\mu^{2}\text{ and }\nabla\varphi.\xi=0.
$$
\end{remark}
\begin{lem}\label{-11}
Let $\psi\in\mathscr{C}^{\infty}(\mathbb{R}^{n},\mathbb{R})$ such that $|\nabla\psi|>0$ in $K$. Then for $\lambda$ large enough $\varphi=e^{\lambda\psi}$ satisfies the sub-ellipticity assumption in $K$.
\end{lem}
\begin{proof}
We can assume that $\psi\ge0$, as we can add a constant to $\psi$ and $\varphi$ is multiplied by a constant $\beta$. 
Changing $\tau $ in $\tau/\beta$ we can see that sub-ellipticity condition is also satisfied for a different constant $C$.
A straightforward calculation shows that
$$
\{q_{2},q_{1}\}(x,\xi,\tau)=4\tau\left({}^{\phantom{0}t}\xi\varphi''\xi+\tau^{2}{}^{\phantom{0}t}(\nabla\varphi)\varphi''\nabla\varphi\right)\label{12}.
$$
Using the fact that $\varphi=e^{\lambda\psi}$ then we have 
$$
\nabla\varphi=\lambda\nabla\psi\varphi,\quad\varphi'_{j}=\lambda\varphi\psi'_{j}\;\text{ and }\;\varphi''_{jk}=\lambda\varphi\psi''_{jk}+\lambda^{2}\varphi\psi'_{j}\psi'_{k},\;1\leq j,k\leq n,
$$
therefore we obtain
$$
\{q_{2},q_{1}\}=4\tau\lambda^{3}\varphi^{3}\left(\lambda\tau^{2}|\nabla\psi|^{4}+\tau^{2}{\phantom{0}}^{t}(\nabla\psi)\psi''\nabla\psi+|\lambda\varphi|^{-2}{}^{\phantom{1}t}\xi\psi''\xi+\lambda^{-1}|\varphi|^{-2}|\nabla\psi.\xi|^{2}\right).
$$
Now if $p_{\varphi}=0$ then $|\xi|^{2}=\tau^{2}|\nabla\varphi|^{2}+\mu^{2}=\tau^{2}\lambda^{2}\varphi^{2}|\nabla\psi|^{2}+\mu^{2}$, which gives that
$$
|\lambda\varphi|^{-2}{}^{\phantom{1}t}\xi\psi''\xi\geq -|\psi''|\left(\tau^{2}|\nabla\psi|^{2}+|\lambda\varphi|^{-2}\mu^{2}\right)\geq -C\tau^{2}\left(|\nabla\psi|^{2}+\lambda^{-2}\right).
$$
Besides, we have
$$
\tau^{2}{\phantom{0}}^{t}(\nabla\psi)\psi''\nabla\psi\geq -C\tau^{2}|\nabla\psi|^{2}.
$$
Then it follows from these estimates that
\begin{align*}
\{q_{2},q_{1}\}&\geq4\tau\lambda^{3}\varphi^{3}\left(\lambda\tau^{2}|\nabla\psi|^{4}+\tau^{2}{\phantom{0}}^{t}(\nabla\psi)\psi''\nabla\psi+|\lambda\varphi|^{-2}{}^{\phantom{1}t}\xi\psi''\xi\right)
\\
&\geq4\tau\lambda^{3}\varphi^{3}(\lambda\tau^{2}|\nabla\psi|^{4}-C\tau^{2}|\nabla\psi|^{2}-C\tau^{2}\lambda^{-2}).
\end{align*}
Since $|\nabla\psi|>0$ in the compact set $K$ then for $\lambda$ large enough  we have $\{q_{2},q_{1}\}\geq C_{\lambda}\tau^3>0$.  As $|\xi|$ is comparable to $\tau$ on $p_\varphi=0$, we obtain the result.
\end{proof}

\begin{lem}\label{13}
Let $f$ and $g$ be two real continuous functions defined  in $ K$ such that $f$ is positive on a compact subset $K$ of $\R^d$ and verifies that
\[
\forall\, y\in K,\qquad f(y)= 0\Longrightarrow g(y)\geq L>0.
\]
We set  $h_{\kappa}=\kappa f+g$, then for $\kappa$ sufficiently large then $h_{\kappa}\geq C$ for some constant $C>0$.
\end{lem}
\begin{proof}
Let $y_0\in K$ to prove the result we distinguish two cases.
\\
\textbf{Case 1: We assume $f(y_0)=0$.} Then according to the assumption made in this lemma we have 
$h_{\kappa}(y_0)=g(y_0)\geq L$. 
Then there exists a neighborhood of $y_0$, $V_{y_0}$ such that for $y\in V_{y_0}$ and every $\kappa>0$,  
$h_{\kappa}(y)\ge g(y)\geq L/2$. Let $\kappa_{y_0}=1$.
\\
\textbf{Case 2:  We assume $f(y_0)>0$.} Since $f$ and $g$  are continuous, there exist $V_{y_0}$ a neighborhood of 
$y_0$ and $C_1,C_2>0$  such that $f(y)\ge C_1$ and $|g(y)|\le C_2$ for all $ y\in V_{y_0}$. Then for all $\kappa\ge (L+ C_2)/C_1$, 
$h_{\kappa}(y)\ge L  $. Let $\kappa_{y_0}=(L+ C_2)/C_1$.

We cover $K$, by compactness argument, by a finite number of such neighborhoods $V_{y_{1}},\dots,V_{y_{p}}$ with associated
$\kappa_{j}=\kappa_{y_j}$. Taking  $\ds\kappa=\max_{1\leq j\leq p}\{\kappa_{j}\}$, we have $h_{\kappa}(y)\ge L/2$ on each $V_{y_{j}}$, then on $K$.
This completes the proof.
\end{proof}
\begin{lem}\label{14}
Let $V$ an open bounded subset of $\mathbb{R}^{d}$ and $\kappa>0$. We suppose that $\varphi$ verifies the sub-ellipticity assumption on $K$  and we set $\rho_{\kappa}=\kappa(q_{2}^{2}+q_{1}^{2})+\tau \{q_{2},q_{1}\}$. Then for $\kappa$ large enough there exists $C>0$ such that for all $(x,\xi)\in K\times\mathbb{R}^{d}$ and $\tau\geq 1$ we have $\rho_{\kappa}(x,\xi)\geq C\langle\xi,\tau\rangle^{4}$.
\end{lem}
\begin{proof}
First we assume $|\xi|$ large with respect $\tau$, that is  $|\xi|\ge \beta\tau$ for $\beta $ sufficiently large, to be fixed below.
We have
\begin{align}
	\label{eq: rho k}
\rho_{\kappa}(x,\xi)&=\kappa(q_{2}^{2}+q_{1}^{2})+\tau \{q_{2},q_{1}\}  \notag
\\
&=\kappa\left(|\xi|^{2}-(\tau^{2}|\nabla\varphi|^{2}+\mu^{2})\right)^{2}+4\kappa\tau^{2}(\nabla\varphi.\xi)^{2}
+4\tau^2{\phantom{0}}^{t}\xi\varphi''\xi+\tau^{4}{\phantom{0}}^{t}(\nabla\varphi)\,\varphi''\nabla\varphi  \notag
\\
&=\kappa\langle\xi,\tau\rangle^{4}\left(1-\frac{(\tau^{2}(1+|\nabla\varphi|^{2})+\mu^{2})}{\langle\xi,\tau\rangle^{2}}\right)^{2}+4\kappa\tau^{2}(\nabla\varphi.\xi)^{2}+4\tau^2{\phantom{0}}^{t}\xi\varphi''\xi
+\tau^{4}{\phantom{0}}^{t}(\nabla\varphi)\,\varphi''\nabla\varphi     
\\
&\geq C'\langle\xi,\tau\rangle^{4}-C\tau^2|\xi|^{2}-C\tau^{4},   \notag
\end{align}
if $\beta $ is sufficiently large such that $\ds\frac{(\tau^{2}(1+|\nabla\varphi|^{2})+\mu^{2})}{\langle\xi,\tau\rangle^{2}}\le 1/2$ and for 
some constants $C',C>0$. If $\beta$ is sufficiently large we obtain $C\tau^2|\xi|^{2}+C\tau^{4}\le C'\langle\xi,\tau\rangle^{4}/2$,
from~\eqref{eq: rho k} we obtain $\rho_{\kappa}(x,\xi) \geq C''\langle\xi,\tau\rangle^{4}$, for $C''>0$. This fixes $\beta$.

Second we assume $|\xi|\le \beta\tau$. As $ \rho_{\kappa}$ is homogeneous of degree 4 in $(\xi,\tau,\mu)$, we can prove the estimate on 
$K'=\{(x,\xi,\tau,\mu)\in K\times\R^d\times [0, +\infty)\times \R, \ |\xi|^2+\tau^2+\mu^2=1,\  |\xi|\le \beta\tau ,\ c_{0}\tau\leq | \mu|\leq c_{0}'\tau \}$ 
taking into account of \eqref{hyp: mu eq tau}. As $K'$ is a compact set, we can apply  Lemma \ref{13} by taking 
$\displaystyle f=q_{2}^{2}+q_{1}^{2}$ and $\displaystyle g=\tau \{q_{2},q_{1}\}$. This completes the proof.
\end{proof}
\begin{thm}\label{21}
Let $\Omega$ be an open bounded set of $\mathbb{R}^{n}$ and $\varphi$ be a function that satisfies the  sub-ellipticity assumption in $K$. Then there exist $\tau_{*}>0$ and $C>0$ such that
\begin{equation}\label{ASCE1}
\tau^{3}\|e^{\tau\varphi}r\|_{0}^{2}+\tau\|e^{\tau\varphi}\nabla r\|_{0}^{2}+\tau^{-1}\sum_{\alpha=2}\|e^{\tau\varphi}D^{\alpha}r\|_{0}^{2}\leq C\|e^{\tau\varphi}Pr\|_{0}^{2}.
\end{equation}
for all $r\in\mathscr{C}_{c}^{\infty}(K)$,  $\tau\geq\tau_{*}$ and $\mu$ satisfying \eqref{hyp: mu eq tau}.
\end{thm}
This theorem is classical, and we can find in H\"ormander~\cite{hormander}. Here we give a proof to be complete.
\begin{proof}
Let's take $\rconj=e^{\tau\varphi}r $ then $Pr=f$ can be written as follow $P_{\varphi}\rconj=g=f\,e^{\tau\varphi}$. Since $g=Q_{2}\rconj+iQ_{1}\rconj$ and $Q_{1}$ and $Q_{2}$ are symmetric then we have
\begin{align}
\|g\|_{0}^{2}&=\|Q_{2}\rconj\|_{0}^{2}+\|Q_{1}\rconj\|_{0}^{2}+i(Q_{1}\rconj,Q_{2}\rconj)-i(Q_{2}\rconj,Q_{1}\rconj)\notag
\\
&=(Q_{2}Q_{2}\rconj,\rconj)+(Q_{1}Q_{1}\rconj,\rconj)+i(Q_{2}Q_{1}\rconj,\rconj)-i(Q_{1}Q_{2}\rconj,\rconj)\notag
\\
&=\Big(\left(Q_{1}^{2}+Q_{2}^{2}+i[Q_{2},Q_{1}]\right)\rconj,\rconj\Big).\label{15}
\end{align}
We fix $\kappa$ large enough such that the statement of Lemma \ref{14} is fulfilled, then for $\tau$ sufficiently large satisfying $\kappa\tau^{-1}\leq 1$ then from \eqref{15} we obtain
\begin{equation}\label{16}
\tau^{-1}\left(\big(\kappa(Q_{1}^{2}+Q_{2}^{2})+i\tau[Q_{2},Q_{1}]\big)\rconj,\rconj\right)\leq\|g\|_{0}^{2}.
\end{equation}
Since the principal symbol of $\displaystyle\kappa(Q_{1}^{2}+Q_{2}^{2})+i\tau[Q_{2},Q_{1}]$ is given by $\rho_{\kappa}(x,\xi,\tau)=\kappa(q_{2}^{2}+q_{1}^{2})+\tau\{q_{2},q_{1}\}$ then from Lemma \ref{14} we have $\rho_{\kappa}(x,\xi,\tau)\geq C\langle\xi,\tau\rangle^{4}$ therefore by G\aa rding inequality (Theorem \ref{ASC8}) it follows that
$$
\re\,\left(\big(\kappa(Q_{1}^{2}+Q_{2}^{2})+	i\tau[Q_{2},Q_{1}]\big)\rconj,\rconj\right)\geq\|\rconj\|_{\tau,2}^{2}.
$$
Combining this inequality with \eqref{16} we find
\begin{equation}
	\label{est: Caleman conj.}
\tau^{-1}\|\rconj\|_{\tau,2}^{2}\leq\|g\|_{0}^{2}.
\end{equation}
which reads
\begin{equation}\label{23}
\tau^{3}\|\rconj\|_{0}^{2}+\tau\|\nabla \rconj\|_{0}^{2}+\tau^{-1}\sum_{|\alpha|=2}\|D^{\alpha}\rconj\|_{0}^{2}\leq C\|e^{\tau\varphi}f\|_{0}^{2}.
\end{equation}
Since we have
\[
e^{\tau\varphi}D_{j}r=(D_{j}+i\tau\partial_{j}\varphi)\rconj,\qquad e^{\tau\varphi}D_{j}D_{k}r=(D_{j}+i\tau\partial_{j}\varphi)(D_{k}+i\tau\partial_{k}\varphi)\rconj,
\]
then we have
\begin{equation}\label{17}
\tau\|e^{\tau\varphi}\nabla r\|_{0}^{2}\leq C\left(\tau^{3}\|\rconj\|_{0}^{2}+\tau\|\nabla \rconj\|_{0}^{2}\right)
\end{equation}
and
\begin{equation}\label{18}
\tau^{-1}\sum_{|\alpha|=2}\|e^{\tau\varphi}D^{\alpha}r\|_{0}^{2}\leq C\left(\tau^{3}\|\rconj\|_{0}^{2}+\tau\|\nabla \rconj\|_{0}^{2}+\tau^{-1}\sum_{|\alpha|=2}\|D^{\alpha}\rconj\|_{0}^{2}\right).
\end{equation}
Thus, estimate \eqref{ASCE1} follows by replacing \eqref{17} and \eqref{18} into \eqref{23}. This concludes the proof.
\end{proof}
%
%
\begin{thm}
	\label{th: Carleman interior H-1}
Let $\Omega$ be an open bounded set of $\mathbb{R}^{d}$, let $K\Subset \Omega$ and $\varphi$ be a function that 
satisfies the  sub-ellipticity assumption in $K$. Then there exist $\tau_{*}>0$ and $C>0$ such that
\begin{equation*}
\tau^3\|e^{\tau\varphi} r\|_0^2 +\tau\|e^{\tau\varphi}\nabla r\|_0 ^2\leq C\tau^{2}\left(\|e^{\tau\varphi}g\|_{0}^{2}
+\|e^{\tau\varphi}f\|_{0}^{2}\right)
\end{equation*}
for all $r \in\mathscr{C}_{c}^{\infty}(K)$ which satisfies \eqref{ASCE2}, $\tau\geq\tau_{*}$ and $\mu$ 
satisfying \eqref{hyp: mu eq tau}.
\end{thm}
From this theorem we can deduce this corollary.
%
%
\begin{cor}
	\label{cor: Carleman interior H-1}
Let $\Omega$ be an open bounded set of $\mathbb{R}^{d}$, let $K\Subset \Omega$ and $\varphi$ be a function that 
satisfies the  sub-ellipticity assumption in $K$. Then there exist $\tau_{*}>0$ and $C>0$ such that
\begin{equation*}
\tau^3\|e^{\tau\varphi} u\|_0^2 +\tau^3\|e^{\tau\varphi} r\|_0^2 +\tau\|e^{\tau\varphi}\nabla r\|_0 ^2\leq C\tau^{2}\left(\|e^{\tau\varphi}g\|_{0}^{2}
+\|e^{\tau\varphi}f\|_{0}^{2}\right)
\end{equation*}
for all $r ,u \in\mathscr{C}_{c}^{\infty}(K)$ which satisfies \eqref{eq: reslov bd}, $\tau\geq\tau_{*}$ and $\mu$ 
satisfying \eqref{hyp: mu eq tau}.
\end{cor}
\begin{proof}
As $r$ satisfies \eqref{ASCE2}, with Theorem~\ref{th: Carleman interior H-1}, we only have to estimate $\tau^3\|e^{\tau\varphi} u\|_0^2 $. From the first equation of 
 \eqref{eq: reslov bd} we have $\tau^3\|e^{\tau\varphi} u\|_0^2\lesssim \tau^3\|e^{\tau\varphi} \mu^{-1}(f+\nabla r) \|_0^2$, which 
 gives the result using  \eqref{hyp: mu eq tau}.
\end{proof}

\begin{proof}[Proof of Theorem \ref{th: Carleman interior H-1}]
We set $P_{\varphi}=e^{\tau\varphi}Pe^{-\tau\varphi}$, $\rconj=e^{\tau\varphi}r$, $F=-e^{\tau\varphi}f$ and $G=i\mu e^{\tau\varphi}g+\tau e^{\tau\varphi}\nabla\varphi.f$. Then from \eqref{ASCE2} we have 
$$
P_{\varphi}\rconj=G+\div(F).
$$
Let $K_1$ be such that $K\Subset K_1\Subset \Omega$ and let $\chi\in \mathscr{C}_{c}^{\infty}(K_1)$ be such that $\chi =1$ on $K$.
Setting $\rconjmoins=\chi\Lambda_{\tau}^{-1}\rconj$ with $\Lambda_\tau=(\tau^2-\Delta)^{1/2}$ and we write
$$
P_{\varphi}\rconjmoins=\chi\Lambda_{\tau}^{-1}P_{\varphi}\rconj+[P_{\varphi},\chi\Lambda_{\tau}^{-1}]\rconj=\chi\Lambda_{\tau}^{-1}(G+\div(F))+[P_{\varphi},\chi\Lambda_{\tau}^{-1}]\rconj,
$$
then we find
\begin{equation}\label{ASCE3}
\|P_{\varphi}\rconjmoins\|_{0}\leq C\left(\tau^{-1}\|G\|_{0}+\|F\|_{0}+\|\rconj\|_{0}\right).
\end{equation}
Applying  Estimate~\eqref{est: Caleman conj.} in the proof of Theorem~\ref{21} to $\rconjmoins$ then we obtain
$$
\tau^{-\frac{1}{2}}\|\rconjmoins\|_{\tau,2}\leq C\|P_{\varphi}\rconjmoins\|_{0}
$$
We have $\rconjmoins=\Lambda_\tau^{-1} \rconj+[\chi ,  \Lambda_\tau^{-1}  ]\rconj$ then 
$\|\rconj \|_{\tau,1}   \lesssim   \|\rconjmoins\|_{\tau,2}  +  \| \rconj \|_0  $.
That together with \eqref{ASCE3} 
$$
\tau^{-\frac{1}{2}}\|\rconj\|_{\tau,1}\leq C\left(\tau^{-1}\|G\|_{0}+\|F\|_{0}+\|\rconj\|_{0}\right).
$$
Multiplying by $\tau$ which is chosen sufficiently large then we obtain
$$
\tau^{\frac{1}{2}}\|\rconj\|_{\tau,1}\leq C\left(\|G\|_{0}+\tau\|F\|_{0}\right)\leq C\tau\left(\|e^{\tau\varphi}g\|_{0}+\|e^{\tau\varphi}f\|_{0}\right).
$$
As $ \|\rconj\|_{\tau,1} $ is equivalent to  $\tau\|e^{\tau\varphi} r\|_0 +\|e^{\tau\varphi}\nabla r\|_0 $ 
arguing as in the proof of \eqref{17} we obtain the result of the theorem.
\end{proof}
\subsection{Local Carleman estimate at the boundary}

In this section because the boundary, we use a tangential pseudo-differential calculus. This calculus is completely analogous 
to the one presented in Section~\ref{sec: Pseudo-differential operators} except that  a function $a(x',x_d,\xi')$ is a symbol 
in $(x',\xi')$  in the sense of Definition~\ref{def: pseudo} where  
$x_d$ is a parameter and the estimates given in Definition~\ref{def: pseudo} are  uniform 
with respect $x_d$. To avoid confusion we denote by $\Ssct^m$ the class of tangential symbol of order 
$m$, $\Opt (a)$ the operator associated with the symbol $a\in\Ssct^m$.
The class of  operators associated with  symbols in $\Ssct^m$ is denoted by $\Psi_{T,\tau}^{m}$.
We refer to \cite{LLR:2020} for details on these symbols and operators.
We consider functions in a half space $\R^{d-1}\times (0,+\infty)=\R^d_+ $, and we denote by
$\| .\|_+=\| .\|_{L^2(\R^d_+)}$ the $L^2$-norm and  $(.,.)_+=(.,.)_{L^2(\R^d_+)}$ the associated inner product. 
At the boundary $x_d=0$ we denote the $L^2$-norm  by $|g|^2=\int_{\R^{d-1}}|g(x')|^2\ud x'$ and the inner product 
associated by $(.,.)_\pd=(.,.)_{L^2(\R^{d-1})}$. A set $W=\omega\times\Gamma$ in $\R^d \times \R^{d-1} \times \R^+$ 
is called a conic open set, 
if there exist $\omega$ an open set in $\R^d$, and $\Gamma $ an open set in $\R^{d-1} \times \R^+$ 
such that for all $(\xi',\tau)\in\Gamma$ and 
$\lambda>0$ then $(\lambda\xi',\lambda\tau)\in\Gamma$. For $s\in\R$ we denote by $\Lsct^{s}$ the tangential operator defined by $\Lsct^{s}=\Opt(\langle\xi',\tau\rangle^{s})$.

We recall the following  microlocal G\aa rding inequality obtained, for instance, by applying sharp G\aa rding inequality.
\begin{thm}[Microlocal G{\aa}rding inequality]
  \label{thm: microlocal Garding}
  Let $K$ be a compact set of $\R^d$ and let $W$ be a conic open set
  of $\R^d \times \R^{d-1} \times \R^+$ contained in
  $K \times \R^{d-1}  \times \R^+$.  Let also $\chi \in \Ssct^0$ be
  homogeneous of order $0$ (for $\left\langle \xi',\tau\right\rangle\geq 1$) and be such that
  $\supp(\chi) \subset W$.

  Let $a(x,\xi',\tau) \in \Ssct^m$, with principal part $a_m$ homogeneous of order $m$. If there exist
  $C_0>0$ and $R>0$ such that
  \begin{equation*}
    \Re a_m(x,\xi',\tau) \geq C_0     \left\langle \xi',\tau\right\rangle^m, \quad (x,\xi',\tau) \in W, 
    \ \tau \in [1,+\infty), \quad    \left\langle \xi',\tau\right\rangle \geq R,
    \end{equation*}
  then
  for any $0<C_1<C_0$, $N \in \N$, there exist $C_N$ and $\tau_\ast\geq 1$
  such that 
  \begin{equation*}
    \Re \big({\Opt(a) \Opt(\chi) u},{ \Opt(\chi) u}\big)_+ \geq C_1 \| \Lsct^{m/2}\Opt(\chi)
      u\|_+ ^2 -C_N \|\Lsct^{-N}u\|_+ ^2, 
  \end{equation*}
for $u \in \SS(\R^d)$ and $\tau \geq \tau_\ast$.
\end{thm}
As we want to change the variables in order to have a flat boundary which is convenient to do the computations, we use 
the language and usual tools of Riemannian geometry.  In this framework  the gradient and divergence operators 
keep   forms we can follow after a change of variables. Our purpose is to use these tools locally and we do not use manifold 
tools as charts, atlas and etc. 
To fix the notation, let $V$ be an open set in $\R^d$. 
Let $g(x)=\big(g_{ij}(x)\big)_{1\le i,j\le d}$ be a positive symmetric matrix called the metric, we denote
by $g^{-1}(x)=\big(g^{ij}(x)\big)_{1\le i,j\le d}$ the inverse of $g(x)$. For a smooth function $r$, we denote by
$(\nablag r(x))^i= \sum_{1\le j\le d}g^{ij}(x)\pd_{x_j}r(x)$ the gradient of $r$. We have $\nablag r(x)\in T_xV$, this means that $\nablag r$ is a 
tangent vector field. 

For $u(x)=\big( u^1(x),\ldots,u^d(x)  \big)$ a smooth  tangent vector field, we define the divergence operator by
\[\div_g u(x)=\big( \det g(x)\big)^{-1/2}  \sum_{1\le j\le d}\pd_{x_j}\Big(\big(\det g(x)\big)^{1/2} u^j(x)\Big).\]
For a smooth function $r$ and a smooth tangent vector field $u$, we have
\begin{equation}
	\label{eq: div product}
\div_g(ru)=r\div_g u+g(\nablag r,u), \text{ where }g(\nablag r,u)=\sum_{1\le i,j\le d} g_{ij}(\nablag r)^iu^j.
\end{equation}
For two smooth functions $r_1$ and $r_2$ we have 
\begin{equation}
	\label{eq: grad product}
\nablag(r_1r_2)=r_1\nablag r_2+r_2\nablag r_1.
\end{equation}
It is well-known that there exist coordinates (called \emph{normal geodesic coordinates}) such that 
the boundary is defined locally by $x_d=0$, the open set $\Omega\cap V$ is defined by $x_d>0$, 
the metric $g$ is such that $g_{id}=g_{di}=0$ for $i=1,\ldots, d-1$ and, $g_{dd}=1$. We denote by 
$\tilde g=(g_{ij})_{1\le i,j\le d-1}$ the metric $g$ on $x_d$ fixed.

We can define on each manifold $x_d= const$ the gradient and divergence operators associated with $\tilde g$ and 
for $r$ a smooth function and $\tilde u=(u^1,\ldots,u^{d-1}) $ a smooth vector field on $x_d= const$, we have
\begin{align*}
& \ (\nablatg r)^i= \sum_{1\le j\le d-1}g^{ij}\pd_{x_j}r \text{ for } i=1,\ldots,d-1,
&\div_{\tilde g}\tilde u=  ( \det \tilde g)^{-1/2}  \sum_{1\le j\le d-1}\pd_{x_j}\Big(( \det \tilde g)^{1/2} u^j\Big) .
\end{align*}
In such coordinates, we have $\det g=\det \tilde g$. The gradient and divergence operators take the following form.
\begin{align}
	\label{def: div grad}
\nablag r=(\nablatg r,\pd_{x_d}r), \  
\div_g u= \div_{\tilde g}\tilde u +\pd_{x_d}u^d+ hu^d, \text{ where }  
h=( \det \tilde g)^{-1/2}  \pd_{x_d}( \det \tilde g)^{1/2}.   
\end{align}
We recall that the equation of the resolvent problem $({\mathcal A}_0 +i\mu )(u,r)=(f,g)$   
locally  takes the form
\begin{equation}
\begin{cases}
	\label{eq: reslov bd 1}
&-\nablag r+ i\mu u=f \text{ in } x_d>0, \\
&-\div_g u+i\mu r=g  \text{ in } x_d>0,\\
&u^d=0  \text{ on } x_d=0.
\end{cases}
\end{equation}
We have the following theorem
 %
 %
 \begin{thm}
 	\label{th: Carleman bord}
Let $x_0\in \R^{d-1}\times \{  0 \}$, we assume there exist a neighborhood of $x_0$ where 
$\varphi$ satisfies \eqref{sub-ellipticity} the sub-ellipticity condition and  $\pd_{x_d} \varphi(x_0)>0$.  Then there exist $ V_0$ be  
an open set of $\R^d$ such that $x_0\in V_0$, $C>0$,   and $ \tau_*>0$ such that
 \begin{equation*}
\tau^{1/2}	| e^{\tau \varphi}r_{|x_d=0}|
	+  \tau^{1/2}  \| e^{\tau \varphi}u\|_+
	+ \tau^{1/2}  \| e^{\tau \varphi} r\|_+   
	+ \tau^{-1/2} \| e^{\tau \varphi} \nablag r\|_+ 
	 \le C\big( 
 \| e^{\tau \varphi}  f\|_+
+ \| e^{\tau \varphi}g\|_+    
\big),
\end{equation*}
for $u,r\in\Con^\infty(\R^d)$ supported on $V_0$, satisfying \eqref{eq: reslov bd 1}, for every $\tau\ge\tau_*$ and $\mu $ satisfying 
\eqref{hyp: mu eq tau}.
 \end{thm}

Let $v=e^{\tau \varphi}u$ and $ \rconj =e^{\tau \varphi} r$.  
We have from \eqref{eq: div product} and \eqref{eq: grad product}
\begin{align*}
&\nablag r= e^{-\tau \varphi} \big( \nablag \rconj -\tau \rconj \nablag\varphi    \big),\\
&\div_gu=e^{-\tau \varphi}\big( \div_g v-\tau g(\nablag \varphi, v) \big).
\end{align*}
Then System~\eqref{eq: reslov bd 1} takes the form
\begin{equation}
\begin{cases}
	\label{eq: reslov bd 2}
&- \nablag \rconj +\tau \rconj \nablag\varphi  + i\mu v=F \text{ in } x_d>0, \\
&-\div_g v+\tau g(\nablag \varphi, v) +i\mu \rconj=G  \text{ in } x_d>0,\\
&v^d=0  \text{ on } x_d=0,
\end{cases}
\end{equation}
where $F= e^{\tau \varphi}f$ and $G=e^{\tau \varphi}g$.

In the following, we denote by $\tilde F= (F^1,\ldots, F^{d-1})$ and by $\tilde v=(v^1,\ldots, v^{d-1})$, then we have $F=(\tilde F,F^d)$ and $v=(\tilde v, v^d)$. Multiplying  \eqref{eq: reslov bd 2}   by $i$, we have
\begin{equation}
\begin{cases}
	\label{eq: reslov bd 3}
&- i\nablatg \rconj +i\tau \rconj \nablatg\varphi  -\mu \tilde v=i\tilde F \text{ in } x_d>0, \\
& -i\pd_{x_d}\rconj +i\tau\rconj \pd_{x_d}\varphi-\mu v^d=iF^d  \text{ in } x_d>0, \\
&-i\div_g v+i\tau \tilde g(\nablatg \varphi, \tilde v)+i\tau v^d\pd_{x_d}\varphi -\mu \rconj=iG  \text{ in } x_d>0,\\
&v^d=0  \text{ on } x_d=0.
\end{cases}
\end{equation}
For this system we prove the following Carleman estimate.
 %
 %
 \begin{prop}
 	\label{prop: carleman loc}
	Let $x_0\in \R^{d-1}\times \{  0 \}$, we assume there exist a neighborhood of $x_0$ where 
$\varphi$ satisfies \eqref{sub-ellipticity} the sub-ellipticity condition and  $\pd_{x_d} \varphi(x_0)>0$.  For $s\in\R$,  there exist $ V_0$ be  an open set 
such that $x_0\in V_0$, $C>0$,   and $ \tau_*>0$ such that
 \begin{equation*}
	|\Lsct^{s+1/2}\rconj_{|x_d=0}|
	+ \tau^{1/2}\| \Lsct^{s}v\|_+
	+ \tau^{-1/2}\| \Lsct^{s+1} \rconj\|_+   
	 \le C\big( 
 \|  \Lsct^{s}  F\|_+
+ \|  \Lsct^{s}  G\|_+    
\big),
\end{equation*}
for $v, \rconj \in \SS(\R^d)$ satisfying \eqref{eq: reslov bd 3}, supported in $V_0$, for every 
$\tau\ge \tau_*$ and $\mu $ satisfying 
\eqref{hyp: mu eq tau}.
 \end{prop}
From this proposition we deduce Theorem~\ref{th: Carleman bord} taking $s=0$. Indeed we have
$\tau^{1/2}| \rconj_{|x_d=0}|\lesssim |\Lsct^{1/2}\rconj_{|x_d=0}|$ and from \eqref{eq: reslov bd 1} we have
\[
\tau^{-1/2}\| e^{\tau\varphi} \nablag r \|_+ \lesssim
\| e^{\tau\varphi} \big( \nablag r- i\mu u  \big)\|_++ |\mu| \tau^{-1/2}\| e^{\tau\varphi}  u  \|_+
 \lesssim   \| e^{\tau\varphi}  f  \|_++ \tau^{1/2} \| e^{\tau\varphi}  u  \|_+,
\]
from \eqref{hyp: mu eq tau}.

We begin by reducing the system in a $2\times 2$ system.
We  denote by
$\zeta'\in \Ssct^1 $ the tangential symbol of the operator $- i\nablatg  +i\tau  \nablatg\varphi $. We have 
\[
\zeta_i=\sum_{1\le j\le d-1}g^{ij}(\xi_j+i\tau\pd_{x_j}\varphi) \text{ for }i=1,\ldots, d-1.
\]
Let  $\Opt (\delta):=-i \div_{\tilde g}+i\tau \tilde g(\nablatg \varphi, \cdot)$ where $\delta\in\Ssct^1$. The principal symbol of the operator $\delta$ is 
$(\xi_1+i\tau\pd_{x_1}\varphi,\ldots, \xi_{d-1}+i\tau\pd_{x_{d-1}}\varphi)$ modulo symbol in $\Ssct^0$.
The first equation of \eqref{eq: reslov bd 3} reads $\Opt (\zeta') \rconj -\mu\tilde v=i\tilde F$.
Applying in both side of this equation 
the operator $\Opt (\delta)$, we obtain
\begin{equation}
\Opt (\delta)\tilde v=   -i\mu^{-1}\Opt (\delta) \tilde F+ \mu^{-1} \Opt (\delta)\Opt (\zeta') \rconj.
\end{equation}
From \eqref{def: div grad} we have 
\begin{align*}
-i\div_g v+i\tau \tilde g(\nablatg \varphi, \tilde v)&= \Opt (\delta)\tilde v-i\pd_{x_d}v^d-i hv^d
\\
&=D_{x_d}v^d-i\mu^{-1}\Opt(\delta)\tilde F+ \mu^{-1} \Opt (\delta)\Opt (\zeta') \rconj-i hv^d.
\end{align*}
From this equation and second and third of \eqref{eq: reslov bd 3} we obtain two equations on $\rconj$ and $v^d$, that is
\begin{equation}
	\label{eq: 2x2}
	\begin{cases}
	& D_{x_d}\rconj +i\tau\rconj \pd_{x_d}\varphi-\mu v^d=iF^d  \text{ in } x_d>0, \\
	& D_{x_d}v^d  + \mu^{-1} \Opt (\delta)\Opt (\zeta') \rconj   -\mu \rconj+i\tau v^d\pd_{x_d}\varphi -i hv^d=iG  +i\mu^{-1}\Opt (\delta) \tilde F\text{ in } x_d>0,\\   
	&v^d=0  \text{ on } x_d=0.
	\end{cases}
\end{equation}
Let $U=(\rconj, v^d)$, the system \eqref{eq: 2x2} has the form
\begin{equation*}
D_{x_d}U+BU=H,  \text{ where } H=(iF^d, iG  +\mu^{-1}\Opt (\delta) \tilde F),
\end{equation*}
and $B$ is a tangential matrix operators with principal symbol
\begin{equation*}
b=
\begin{pmatrix}
i\tau \pd_{x_d}\varphi  & -\mu\\
\mu^{-1} q(x,\xi')-\mu & i\tau  \pd_{x_d}\varphi 
\end{pmatrix},
\end{equation*}
modulo $\mu^{-1}\Ssct^1$, 
where $q(x,\xi')= \ds \sum_{1\le i,j\le d-1}g^{ij}(x)\big(\xi_i+i\tau \pd_{x_i}\varphi(x)\big)\big(\xi_j+ i\tau\pd_{x_j}\varphi(x)\big)$.
The characteristic polynomial of $b$ is given by $P(\lambda)=(\lambda -i\tau \pd_{x_d}\varphi )^2+q-\mu^2$.
Let $\alpha\in\mathbb{C}$ such that $\alpha^2= q-\mu^2$ with $\Re \alpha \ge0$. The definition of $\alpha$ is ambiguous
when $q- \mu^2\le 0$ but in this case if $q- \mu^2< 0$ the root are simple and the analysis below is independent of the choice 
of root. In particular the roots are smooth, or if  $q- \mu^2= 0$ the root is double and below, we give a specific analysis in this case.
The root of $P(\lambda)$ are $i\tau \pd_{x_d}\varphi\pm i\alpha$ and the analysis in what follows depends on the location 
of roots in complex plane.
We have the following result, denoting $s=t^2$ where $t,s\in{\mathbb C}$ we have for $r_0>0$,
\begin{equation}
	\label{claim: position carre}
	|\Re t|  \lesseqqgtr r_0\iff 4r_0^2\Re s-4r_0^4+(\Im s)^2 \lesseqqgtr 0.
\end{equation}
Indeed, let $t=x+iy$, we have $\Re s= x^2-y^2$ and $\Im s=2xy$, we obtain 
$4r_0^2\Re s-4r_0^4+(\Im s)^2=4(r_0^2+y^2)(x^2-r_0^2)$ which gives the result. 

From \eqref{claim: position carre}, we obtain that $|\Re \alpha| \lesseqqgtr   \tau | \pd_{x_d}\varphi |$  is equivalent to 
\begin{equation}
	\label{positivity condition}
4\tau^2(  \pd_{x_d}\varphi )^2(\Re q-\mu^2)-4\tau^4(  \pd_{x_d}\varphi )^4+(\Im q)^2  \lesseqqgtr 0,
\end{equation}
where, from the definition of $q$, we have 
\begin{equation*}
\left\{\begin{array}{l}
\ds\Re q (x,\xi')=\sum_{1\le i,j\le d-1}g^{ij}(x)\big(\xi_i\xi_j-\tau^2 \pd_{x_i}\varphi(x)\pd_{x_j}\varphi(x)\big),\\
\ds\Im q (x,\xi')=\tau \sum_{1\le i,j\le d-1}g^{ij}(x) \xi_j\pd_{x_i}\varphi(x).
\end{array}\right.
\end{equation*}

We prove a microlocal Carleman estimate.
 %
 %
 \begin{lem}
 	\label{lem: Carleman microloc}
	 Let $x_0\in \R^{d-1}\times \{  0 \}$, we assume there exist a neighborhood of $x_0$ where 
$\varphi$ satisfies \eqref{sub-ellipticity} the sub-ellipticity condition and  $\pd_{x_d} \varphi(x_0)>0$.  Let $(\xi_0',\tau_0)\in\R^{d-1}\times \R^+ $ be 
such that $|\xi_0'|^2+\tau_0^2=1$. There exist $ W$ be  an open conic set of $(x_0,\xi_0',\tau_0)$,
 $\chi_1 \in \Ssct^0$ be an 
homogenous symbol of order 0  for $\langle \xi',\tau\rangle\ge 1$  supported in $W$ and $\chi_1=1$ in a conic 
neighborhood of $(x_0,\xi'_0,\tau_0)$.  For $s\in\R$,  
there exist $C>0$,   and $ \tau_*>0$ such that
 \begin{multline*}
	|\Lsct^{s+1/2}\Opt(\chi_1)\rconj_{|x_d=0}|
	+\tau ^{1/2} \| \Lsct^{s}\Opt(\chi_1)v\|_+
	+ \tau^{-1/2} \| \Lsct^{s+1}\Opt(\chi_1)\rconj\|_+\\
	 \le C\big( 
 \|  \Lsct^{s}  F \|_+
+ \|  \Lsct^{s}  G\|_+      + \|  \Lsct^{s}    \rconj  \|_+
 + \| \Lsct^{s}   v^d  \|_+
\big),
\end{multline*}
for $v, \rconj \in \SS(\R^d)$ satisfying \eqref{eq: reslov bd 3},    
for every 
$\tau\ge \tau_*$ and $\mu $ satisfying 
\eqref{hyp: mu eq tau}.
 \end{lem}
 This lemma implies Proposition  \ref{prop: carleman loc}
  as we can cover $\langle \xi',\tau\rangle= 1$ by a finite number of 
 open sets given by the statement of Lemma~\ref{lem: Carleman microloc}.
 
For the proof of Lemma~\ref{lem: Carleman microloc}, we distinguish two cases, $\alpha\ne0$ and $\alpha=0$.

Assume  that $\alpha(x_0,\xi'_0,\tau_0)\ne 0$.
By continuity and homogeneity in $(\xi',\tau)$, $\alpha\ne0$   in a conic neighborhood $W$
of $(x_0,\xi'_0,\tau_0)$.  Let $\chi_0 \in \Ssct^0$ be an 
homogenous symbol of order 0  for $\langle \xi',\tau\rangle\ge 1$ such that $\chi_0 =1 $ in a 
conic neighborhood of $(x_0,\xi'_0,\tau_0)$, supported in $W$ and $\chi_1$ supported on $\chi_0=1$. 
Writing
\[
b=
\begin{pmatrix}
i\tau \pd_{x_d}\varphi  & -\mu\\
\mu^{-1}\alpha^2 & i\tau  \pd_{x_d}\varphi 
\end{pmatrix},
\]
the left eigenvector associated with $i\tau \pd_{x_d}\varphi  +i\alpha$ (resp.  $i\tau \pd_{x_d}\varphi  -i\alpha$ ) is 
$\begin{pmatrix}
-i\alpha  &   \mu
\end{pmatrix}$ \big(resp. 
$\begin{pmatrix}
 i\alpha    & \mu
\end{pmatrix}$\big).

Let $\tilde \alpha= \chi_0 \alpha$, as $\alpha$ is a smooth homogenous function of order 1 in $W$, $\tilde \alpha \in\Ssct^1$.
Recall the notation $\Lsct^s=\Opt(\langle \xi',\tau\rangle^s)$, according with the above algebraic computations and with the left vector found, we define 
\begin{equation}
	\label{def: z1 z2}
	\left\{\begin{array}{l}
	z_1=-i\Lsct^{-1}\Opt(\tilde \alpha) \Opt(\chi_1)w+\mu  \Lsct^{-1}\Opt(\chi_1) v^d\\
	z_2=i\Lsct^{-1}\Opt(\tilde \alpha) \Opt(\chi_1)w+\mu  \Lsct^{-1}\Opt(\chi_1) v^d.
	\end{array}\right.
\end{equation}
As $v^d=0$ on $x_d=0$ we obtain $z_1+z_2=0$ on $x_d=0$. Applying $\pm i\Lsct^{-1}\Opt(\tilde \alpha) \Opt(\chi_1)$ to the first equation ~\eqref{eq: 2x2},  $\mu\Lsct^{-1}  \Opt(\chi_1) $ to the second equation and summing up, we obtain
\begin{multline}
	\label{eq: on z, large xi}
D_{x_d} z_j + \Opt\big(  i\tau \pd_{x_d}\varphi  +(-1)^ji\tilde\alpha  \big)z_j = H_j \text{ where} \\
\| \Lsct^{s}  H_j\|_+  \lesssim  
 \|  \Lsct^{s}  F^d \|_+
+ \|  \Lsct^{s}  G\|_+   
+\| \Lsct^{s}  \tilde F  \|_+
 + \|  \Lsct^{s}    \rconj  \|_+
 + \| \Lsct^{s}   v^d  \|_+. 
\end{multline}
 We compute
\begin{align}
	\label{eq: mult. 1}
2\Re( H_j, i\Lsct^{2s+1}z_j)_+
&= 2\Re(  D_{x_d} z_j + \Opt\big(  i\tau \pd_{x_d}\varphi  +(-1)^ji\tilde\alpha  \big)z_j ,  i\Lsct^{2s+1}z_j)_+\\
&= |\Lsct^{s+1/2}(z_j)_{|x_d=0}|^2 +2\Re( \Lsct^{2s+1} \Opt\big(  \tau \pd_{x_d}\varphi  +(-1)^j\tilde\alpha  \big)z_j,z_j)_+,  \notag
\end{align}
using that 
\begin{equation}
	\label{for: integ parts}
2\Re(  D_{x_d} h, i\Lsct^{2m}h)_+= | \Lsct^m h_{|x_d=0}  |^2,
\end{equation}
for $h\in  \SS(\R^d) $.

If $j=2$, we have $  \tau \pd_{x_d}\varphi  +\Re \alpha \gtrsim \tau +|\xi'| $ in $W$. Let $\chi_2\in\Ssct^0$ supported in $\chi_0=1$
and $\chi_2=1$ on the support of $\chi_1$. From symbolic calculus we have 
\begin{equation}
	\label{eq: z microloc}
\| \Lsct^s \big(  z_2-\Opt (\chi_2)z_2 \big)\|_+\lesssim \| \Lsct^{-N}\rconj \|_++ \| \Lsct^{-N} v^d \|_+.
\end{equation}
Then the tangential G\aa rding inequality 
of Theorem~\ref{thm: microlocal Garding} applies and we have
\[
2\Re( \Lsct^{2s+1} \Opt\big(  \tau \pd_{x_d}\varphi +\tilde\alpha  \big)z_2,z_2)_+\ge C_1 \| \Lsct^{s+1}z_2\|_+^2 
-C_N\big(
 \| \Lsct ^{-N }w\|_+^2  +   \|  \Lsct ^{-N } v^d \|_+ ^2
 \big).
\]
From \eqref{eq: mult. 1}, we then deduce
\begin{multline}
	\label{eq: mult 1 bis}
2\Re( H_2, i\Lsct^{2s+1}z_2)_+\ge \\
C_1 \big( |\Lsct^{s+1/2}(z_2)_{|x_d=0}|^2+ \| \Lsct^{s+1}z_2\|_+^2 \big)
-C_N\big(
 \| \Lsct ^{-N }w\|_+^2  +   \|  \Lsct ^{-N } v^d \|_+ ^2
 \big),
\end{multline}
for $C_1>0$,  for every $N>0$ and  $C_N>0$, uniformly with respect to $\tau$ chosen sufficiently large.
This implies 
\begin{align}
	\label{eq: mult. 2}
 |\Lsct^{s+1/2}(z_2)_{|x_d=0}|+ \| \Lsct^{s+1}z_2\|_+\lesssim \|\Lsct^{s}   H_2  \|_+  + \| \Lsct ^{-N }w\|_+  +   \|  \Lsct ^{-N } v^d \|_+ .
 \end{align}
 %
 %
 \begin{lem}
 	\label{lem: cas alpha diff zero}
Assume that $\alpha\ne0$ in $W$.

If $\Re \alpha -  \pd_{x_d}\varphi  \ne 0$ on $W$, we have
\begin{align}
	\label{eq: est. simple roots}
 \| \Lsct^{s+1} z_1\|_+\le C\big(  \| \Lsct^{s} H_1\|_+   +| \Lsct ^{s+1/2} (z_1)_{|x_d=0}| 
 + \|\Lsct^{s} \rconj \|_+ +  \|\Lsct^{s} v^d \|_+ 
  \big),
\end{align}   
for some $C>0$.

If $\Re \alpha - \tau  \pd_{x_d}\varphi  =0$ at $(x_0,\xi'_0,\tau_0)$, we have
\begin{align}
	\label{eq: est. simple roots 2}
 \| \Lsct^{s+1/2} z_1\|_+\le C\big(  \| \Lsct^{s} H_1\|_+   +| \Lsct ^{s+1/2} (z_1)_{|x_d=0}| 
 + \|\Lsct^{s} \rconj \|_+ +  \|\Lsct^{s} v^d \|_+ 
  \big),
\end{align}   
for some $C>0$.
 \end{lem}
\begin{proof}
We have to distinguish three cases, that is  
$|\Re \alpha|   \lesseqqgtr \tau | \pd_{x_d}\varphi |$ at $(x_0,\xi'_0,\tau_0)$. 
\begin{itemize}
\item[$\bullet$]
If $|\Re \alpha|<\tau | \pd_{x_d}\varphi |$, from \eqref{positivity condition} this is equivalent to 
\begin{equation*}
4\tau^2(  \pd_{x_d}\varphi )^2(\Re q-\mu^2)-4\tau^4(  \pd_{x_d}\varphi )^4+(\Im q)^2<0.
\end{equation*}
 We have $  \tau \pd_{x_d}\varphi  -\Re \alpha \gtrsim \tau +|\xi'| $ in $W$. Then we have the same computations 
 as in~\eqref{eq: mult 1 bis} and \eqref{eq: mult. 2}, and we have 
 \begin{align*}
 |\Lsct^{s+1/2}(z_1)_{|x_d=0}|+ \| \Lsct^{s+1}z_1\|_+\lesssim \|\Lsct^{s}   H_1  \|_+  + \| \Lsct ^{-N }w\|_+  +   \|  \Lsct ^{-N } v^d \|_+
\end{align*}
which is a better estimate than \eqref{eq: est. simple roots}.
\item[$\bullet$]
 If $|\Re \alpha|>\tau | \pd_{x_d}\varphi |$, from \eqref{claim: position carre} this is equivalent to 
 \begin{equation*}
4\tau^2(  \pd_{x_d}\varphi )^2(\Re q-\mu^2)-4\tau^4(  \pd_{x_d}\varphi )^4+(\Im q)^2>0.
\end{equation*}
Observe that  this case contains the case where $\tau_0=0$ as $|\xi'_0|=1$, and in $W$ we have 
$q(x,\xi', \tau)\ge c|\xi'|^2$ and $|\xi'|\gg\tau$.

As  $ -\tau \pd_{x_d}\varphi  +\Re \alpha \gtrsim \tau +|\xi'| $ in $W$, from~\eqref{eq: mult. 1}, we can introduce a cutoff as 
in \eqref{eq: z microloc} to apply  the tangential G\aa rding inequality 
of Theorem~\ref{thm: microlocal Garding}, we deduce
\begin{align*}
-2\Re( H_1, i\Lsct^{2s+1}z_1)_+ +  |\Lsct^{s+1/2}(z_1)_{|x_d=0}|^2\gtrsim \| \Lsct^{s+1}z_1\|_+^2-C_N\big(
 \| \Lsct ^{-N }w\|_+^2  +   \|  \Lsct ^{-N } v^d \|_+ ^2
 \big),
\end{align*}
and then
\begin{align*}
 \| \Lsct^{s+1}z_1\|_+\lesssim \|\Lsct^{s}   H_1  \|_++  |\Lsct^{s+1/2}(z_1)_{|x_d=0}| +\| \Lsct ^{-N }w\|_+  +   \|  \Lsct ^{-N } v^d \|_+ ,
\end{align*}
which implies \eqref{eq: est. simple roots}.
\item[$\bullet$]
If $\Re \alpha =\tau  \pd_{x_d}\varphi $ at $(x_0,\xi'_0,\tau_0)$, from \eqref{claim: position carre}  and as $\Re \alpha$ and 
$ \pd_{x_d}\varphi $ are positive, this is equivalent to 
\begin{equation*}
4\tau^2(  \pd_{x_d}\varphi )^2(\Re q-\mu^2)-4\tau^4(  \pd_{x_d}\varphi )^4+(\Im q)^2=0, \text{ at } (x_0,\xi'_0,\tau_0).
\end{equation*}
\end{itemize}
We use Carleman technics to obtain an estimate. Before doing that we must translate sub-ellipticity assumption 
\eqref{sub-ellipticity} on $p_\varphi$ on analogous condition on $\alpha$. 
First observe that
\[
p_\varphi(x,\xi,\tau)=(\xi_d+i\tau  \pd_{x_d}\varphi )^2
+\alpha^2= (\xi_d+i\tau  \pd_{x_d}\varphi +i\alpha) (\xi_d+i\tau  \pd_{x_d}\varphi -i\alpha).
\]
As $i\tau_0  \pd_{x_d}\varphi(x_0) -i\alpha(x_0,\xi'_0,\tau_0)\in\R$, $p_\varphi=0$ is equivalent to 
$\xi_d+i\tau  \pd_{x_d}\varphi -i\alpha=0$. Noting that $i\tau  \pd_{x_d}\varphi +i\alpha\not\in\R$ 
thus $\xi_d+i\tau  \pd_{x_d}\varphi +i\alpha \ne0$ in $W$.
Second, for a smooth function $q=q_r+iq_i$
where $q_r,q_i$ are real valued, we have $\{q,\bar q\}= 2i\{q_i , q_r\}$.
Thus on $p_\varphi=0$ we have 
\begin{align*}
\{ p_\varphi ,\overline{p_\varphi}  \}&=| \xi_d+i\tau  \pd_{x_d}\varphi +i\alpha |^2
\{ \xi_d+i\tau  \pd_{x_d}\varphi -i\alpha , \xi_d-i\tau  \pd_{x_d}\varphi +i\bar\alpha \}\\
&= 2i| \xi_d+i\tau  \pd_{x_d}\varphi +i\alpha |^2
\{ \tau  \pd_{x_d}\varphi -\Re \alpha , \xi_d+ \Im \alpha \}.
\end{align*}
Thus sub-ellipticity condition reads in $W$, there exists $C>0$ such that
\begin{align}
	\label{sub ellipticity, alpha}
\xi_d+i\tau  \pd_{x_d}\varphi -i\alpha=0 \Longrightarrow  \{ \xi_d+ \Im \alpha,  \tau 
 \pd_{x_d}\varphi -\Re \alpha  \}\ge C\langle \xi',\tau\rangle.
\end{align}
At $(x_0,\xi'_0,\tau_0)$, observe that we can choose $\xi_d $ such that $\xi_d+ \Im \alpha=0$ and as 
$  \tau_0  \pd_{x_d}\varphi -\Re \alpha=0$, the condition \eqref{sub ellipticity, alpha} means, by continuity and homogeneity,
 there exists $C>0$ such that
\begin{equation}
	\label{positivity bracket}
 \{ \xi_d+ \Im \alpha,  \tau  \pd_{x_d}\varphi -\Re \alpha  \}\ge C\langle \xi',\tau\rangle \text{ in } W,
\end{equation}
eventually shrinking $W$.

Let 
\begin{align*}
&A=\frac12\big(   \Opt  (i\tau \pd_{x_d}\varphi  -i\tilde\alpha )+ \Opt  (i\tau \pd_{x_d}\varphi  -i\tilde\alpha )^*\big), \\
&B=\frac1{2i}\big(   \Opt  (i\tau \pd_{x_d}\varphi  -i\tilde\alpha )- \Opt  (i\tau \pd_{x_d}\varphi  -i\tilde\alpha )^*\big).
\end{align*}
We have $A=A^*$,  $B=B^*$,  $ \Opt  (i\tau \pd_{x_d}\varphi  -i\tilde\alpha )=A+iB$, and principal symbol of $A$ is 
$\Im \tilde\alpha $
and principal symbol of $B$ is $ \tau \pd_{x_d}\varphi  -\Re \tilde\alpha$.

Now from~\eqref{eq: on z, large xi}  we compute for  $z=\Opt(\chi_0)\Lsct^{s}  z_1$ 
\begin{align}
	\label{eq: square of sum}
 \|  \big( D_{x_d}  + \Opt  (i\tau \pd_{x_d}\varphi  -i\tilde\alpha ) \big)z   \|_+^2
 &=  \|   \big( D_{x_d}  + A \big)z   \|_+^2   
+ \|  Bz   \|_+^2   
\\ & \quad +  2\Re \Big(   \big(  D_{x_d}  + A \big)z  ,    i  B z   \Big)_+.\notag
\end{align}
We have  
\begin{equation*}
2\Re   \big(  D_{x_d}  z ,    i  B z   \big)_+=   \big( [ D_{x_d},  i  B ]z , z \big)_++ (Bz_{|x_d=0},z_{|x_d=0})_\pd .
\end{equation*}
As the principal symbol of $B$ is $ \tau \pd_{x_d}\varphi  -\Re \tilde\alpha$, we obtain
\begin{multline}
	\label{eq: IPP dx B}
2\Re   \big(  D_{x_d}  z ,    i  B z   \big)_+
\ge \Re \big(i  [    D_{x_d} ,    \Opt  (\tau \pd_{x_d}\varphi  -\Re \tilde\alpha ) ]z, z  \big)_+ \\
\quad+ \Re ( \Opt( \tau \pd_{x_d}\varphi  -\Re \tilde\alpha  )z_{|x_d=0} , z_{|x_d=0})_\pd - C\| z\|_+^2-C |z_{|x_d=0}|^2,
\end{multline}
for some constant $C>0$.
We also have
\begin{align*}
 2\Re \big(    A z  ,    i  B z   \big)_+= \big(  i  [  A,B] z  ,   z   \big)_+\ge\Re  \big(i [   \Opt  ( \Im \tilde\alpha )  ,   
 \Opt  (\tau \pd_{x_d}\varphi  -\Re \tilde\alpha ) ]z   \big)_+ -C\| z\|_+^2.
\end{align*}
Then from this estimate and \eqref{eq: IPP dx B}, we obtain
\begin{multline}
	\label{eq: inner pd}
  2\Re \Big(   \big(  D_{x_d}  + A \big)z  ,    i  B z   \Big)_+
 \ge
 \Re \big(i  [    D_{x_d} + \Opt  ( \Im \tilde\alpha )
 ,    \Opt  (\tau \pd_{x_d}\varphi  -\Re \tilde\alpha ) ]z, z  \big)_+
\\
 + \Re ( \Opt( \tau \pd_{x_d}\varphi  -\Re \tilde\alpha  )z_{|x_d=0} , z_{|x_d=0})_\pd  - C\| z\|_+^2-C |z_{|x_d=0}|^2.
\end{multline}
The principal symbol of $i [    D_{x_d} + \Opt  ( \Im \tilde\alpha ) ,    \Opt  (\tau \pd_{x_d}\varphi  -\Re \tilde\alpha ) ]$ is 
$ \{ \xi_d+ \Im \alpha,  \tau  \pd_{x_d}\varphi -\Re \alpha  \}$,
 then from \eqref{positivity bracket} and microlocal 
G\aa rding inequality of Theorem~\ref{thm: microlocal Garding}, we have
\begin{align}
	\label{eq: bracket +}
  \Re \big(i  [    D_{x_d} + \Opt  ( \Im \tilde\alpha ) 
 ,    \Opt  (\tau \pd_{x_d}\varphi  -\Re \tilde\alpha ) ]z, z  \big)_+ \ge C_1 \| \Lsct^{1/2} z\|_+^2 -  C_N\| \Lsct^{-N} z_1\|_+^2 .
\end{align}
We have 
\[
 |( \Opt( \tau \pd_{x_d}\varphi  -\Re \tilde\alpha  )z_{|x_d=0} , z_{|x_d=0})_\pd|\lesssim | \Lsct ^{1/2} z_{|x_d=0}|^2,
\]
then from~\eqref{eq: square of sum}, \eqref{eq: inner pd} and \eqref{eq: bracket +} we obtain
\begin{align}
	\label{eq: cas racine reelle 1}
 \| \Lsct^{1/2} z\|_+ 
 \lesssim
 \big\| \big( D_{x_d}  + \Opt  (i\tau \pd_{x_d}\varphi  -i\tilde\alpha ) \big)z   \big\|_+
 + | \Lsct ^{1/2} z_{|x_d=0}| +\| \Lsct^{-N} z_1\|_+,
\end{align}
as we can absorb the remainder term $\| z\|_+^2$ by the left hand side.
Recalling  the definition of $z_1$ given by formula~\eqref{def: z1 z2}, the 
symbolic calculus yields
\[
\|\Lsct^{1/2}\Opt(\chi_0)\Lsct^{s}z_1-\Lsct^{s+1/2}z_1  \|_+\lesssim \|\Lsct^{s-1/2} \rconj \|_+ +  \|\Lsct^{s-1/2} v^d \|_+.
\]
From  $z=\Opt(\chi_0)\Lsct^{s}  z_1$,  we deduce 
\begin{align}
\|\Lsct^{s+1/2}z_1\|_+\lesssim \| \Lsct^{1/2} z\|_+  + \|\Lsct^{s-1/2} \rconj \|_+ +  \|\Lsct^{s-1/2} v^d \|_+.
\end{align}
Symbolic calculus also gives 
\begin{align}
 \big\| \big( D_{x_d}  + \Opt  (i\tau \pd_{x_d}\varphi  -i\tilde\alpha ) \big)z   \big\|_+\lesssim \| \Lsct ^{s}H_1\|_+
 + \|\Lsct^{s} \rconj \|_+ +  \|\Lsct^{s} v^d \|_+,
\end{align}
and
\begin{align}
	\label{eq: cas racine reelle 2}
 | \Lsct ^{1/2} z_{|x_d=0}| \lesssim  | \Lsct ^{s+1/2}( z_1)_{|x_d=0}|  .
\end{align}
Then from \eqref{eq: cas racine reelle 1}--\eqref{eq: cas racine reelle 2} we obtain
\begin{align}
\|\Lsct^{s+1/2}z_1\|_+\lesssim  \| \Lsct ^{s}H_1\|_++ | \Lsct ^{s+1/2}( z_1)_{|x_d=0}|
 + \|\Lsct^{s} \rconj \|_+ +  \|\Lsct^{s} v^d \|_+,
\end{align}
which is  \eqref{eq: est. simple roots 2}. 
This achieves the proof of Lemma~\ref{lem: cas alpha diff zero} as we have treated the three cases.
\end{proof}
We can prove Lemma \ref{lem: Carleman microloc} in the case $\alpha\ne0$.

If $\Re \alpha -  \pd_{x_d}\varphi  \ne 0$ on $W$,
from \eqref{eq: mult. 2},     Lemma~\ref{lem: cas alpha diff zero},
and as $z_1+z_2=0$ on  $x_d=0$, we deduce
\begin{multline*}
	|\Lsct^{s+1/2}(z_2)_{|x_d=0}|+ \| \Lsct^{s+1}z_2\|_++ \| \Lsct^{s+1}z_1\|_+    \\
	\lesssim \|\Lsct^{s}   H_2  \|_++ \|\Lsct^{s}   H_1  \|_+
	 + \|\Lsct^{s} \rconj \|_+ +  \|\Lsct^{s} v^d \|_+ .
\end{multline*}
From \eqref{eq: on z, large xi} we deduce
\begin{align}
	\label{eq: mult. 4}
	|\Lsct^{s+1/2}(z_2)_{|x_d=0}|+ \| \Lsct^{s+1}z_2\|_++ \| \Lsct^{s+1}z_1\|_+
	& \lesssim  
 \|  \Lsct^{s}  G\|_+     +  \|  \Lsct^{s}  F^d \|_++  \| \Lsct^{s}  \tilde F  \|_+\\
 &  \quad 
 + \| \Lsct^{s}   v^d  \|_+
  + \|  \Lsct^{s}    \rconj  \|_+   . \notag
\end{align}
We have from  \eqref{def: z1 z2}, $z_1+z_2= 2\mu \Lsct^{-1}\Opt(\chi_1)v^d$ and from \eqref{hyp: mu eq tau}
we deduce 
\begin{equation}  
\tau\| \Lsct^{s}\Opt(\chi_1)v^d\|_+\lesssim |\mu| \| \Lsct^{s}\Opt(\chi_1)v^d\|_+\lesssim   \| \Lsct^{s+1}z_1\|_+ +\| \Lsct^{s+1}z_2\|_+.
\end{equation}
We have $\Opt(\tilde\alpha)^*\Lsct^{s}\Lsct^{s}\Opt(\tilde\alpha)= \Opt(\langle \xi',\tau \rangle^{2s}\tilde\alpha^2)$ 
modulo an operator of order $2s+1$.
As $\tilde\alpha$ is not $0$ on the support of $\chi_1$, the tangential G\aa rding inequality 
of Theorem~\ref{thm: microlocal Garding} yields
\begin{equation*}
\| \Lsct^{s}\Opt(\tilde\alpha)\Opt(\chi_1)\rconj\|_+ + \|\Lsct^{-N}\rconj\|_+ \gtrsim \| \Lsct^{s}\Opt(\chi_1)\rconj\|_+ ,
\end{equation*}
for every $N>0$.
From this and as $z_2-z_1=2i\Lsct^{-1}\Opt(\tilde\alpha)\Opt(\chi_1)\rconj$ from \eqref{def: z1 z2}, we deduce
\begin{equation} 
 \| \Lsct^{s+1}\Opt(\chi_1)\rconj\|_+ \lesssim   \| \Lsct^{s+1}z_1\|_+ +\| \Lsct^{s+1}z_2\|_+ + \|\Lsct^{-N}\rconj\|_+.
\end{equation}
From the first equation of \eqref{eq: reslov bd 3} and from \eqref{hyp: mu eq tau} we have 
\begin{align}
\tau \| \Lsct^{s}\Opt (\chi_1) \tilde v\|_+
&\lesssim \| \Lsct^{s+1} \Opt (\chi_1) \rconj\|_++ \| \Lsct^{s} \tilde F\|_+ 
+ \| \Lsct^{s}   \rconj\|_+ \nonumber \\
&  \lesssim   \| \Lsct^{s+1}z_1\|_+ +\| \Lsct^{s+1}z_2\|_+ + \| \Lsct^{s} \tilde F\|_+ 
+ \| \Lsct^{s}   \rconj\|_+. 
\end{align}
From \eqref{def: z1 z2}, $(z_2)_{|x_d=0}=    i\Lsct^{-1}\Opt(\tilde\alpha)\Opt(\chi_1)\rconj_{|x_d=0} $, 
arguing as from above and using the G\aa rding estimate of Theorem \ref{ASC8},
we have 
\begin{equation}
	\label{eq: est trace microloc conj}
| \Lsct^{s+1/2} \Opt(\chi_1) \rconj_{|x_d=0}  |\lesssim | \Lsct^{s+1/2}  (z_2)_{|x_d=0} |.
\end{equation}
 From \eqref{eq: mult. 4}--\eqref{eq: est trace microloc conj}  we obtain Lemma \ref{lem: Carleman microloc}.

 If $\Re \alpha - \tau  \pd_{x_d}\varphi  =0$ at $(x_0,\xi'_0,\tau_0)$,  adding 
 \eqref{eq: mult. 2} to $\varepsilon$\eqref{eq: est. simple roots 2} for $\varepsilon>0$, we deduce
\begin{multline*}
	|\Lsct^{s+1/2}(z_2)_{|x_d=0}|+ \| \Lsct^{s+1/2}z_2\|_++ \varepsilon\| \Lsct^{s+1/2}z_1\|_+    \\
	\lesssim \|\Lsct^{s}   H_2  \|_++ \varepsilon\|\Lsct^{s}   H_1  \|_+
	 + \|\Lsct^{s} \rconj \|_+ +  \|\Lsct^{s} v^d \|_++ \varepsilon|\Lsct^{s+1/2}(z_1)_{|x_d=0}|.
\end{multline*}
From \eqref{eq: on z, large xi} and as $z_1+z_2=0$ on  $x_d=0$, we deduce for $\varepsilon$ small enough tat
\begin{multline}
	\label{eq: mult. 4Bis}
	|\Lsct^{s+1/2}(z_2)_{|x_d=0}|+ \| \Lsct^{s+1/2}z_2\|_++ \| \Lsct^{s+1/2}z_1\|_+   \\
	 \lesssim  
 \|  \Lsct^{s}  G\|_+     +  \|  \Lsct^{s}  F^d \|_++  \| \Lsct^{s}  \tilde F  \|_+
  + \| \Lsct^{s}   v^d  \|_+
  + \|  \Lsct^{s}    \rconj  \|_+   . 
\end{multline}
We have from  \eqref{def: z1 z2}, $z_1+z_2= 2\mu \Lsct^{-1}\Opt(\chi_1)v^d$. Let $\chi_2\in\Ssct^0$ supported in $\chi_0=1$
and $\chi_2=1$ on the support of $\chi_1$. From symbolic calculus we have 
\begin{equation*}
\|(\Opt (\chi_2)  \Lsct^{s+3/2}\mu^{-1} )( \mu \Lsct^{-1}\Opt(\chi_1)v^d)-  \Lsct^{s+1/2}\Opt(\chi_1)v^d
\|_+\lesssim  \| \Lsct^{-N} v^d \|_+.
\end{equation*}
As $\Opt (\chi_2)  \Lsct^{s+3/2}\mu^{-1} $ is an operator of order $s+1/2$ as $|\mu|$ and $|\xi'|$ are comparable on the support
of $\chi_2$, 
 we deduce
\begin{equation}  
\tau^{1/2}\| \Lsct^{s}\Opt(\chi_1)v^d\|_+\lesssim \| \Lsct^{s+1/2}\Opt(\chi_1)v^d\|_+\lesssim   \| \Lsct^{s+1/2}z_1\|_+ +\| \Lsct^{s+1/2}z_2\|_++\| \Lsct^{-N} v^d \|_+.
\end{equation}
We have $\tau^{-1}\Opt(\tilde\alpha)^*\Lsct^{s}\Lsct^{s}\Opt(\tilde\alpha)= \tau^{-1}\Opt(\langle \xi',\tau \rangle^{2s}\tilde\alpha^2)$ 
modulo an operator of order $2s$.
As $\tilde\alpha$ is not $0$ on the support of $\chi_1$, the tangential G\aa rding inequality 
of Theorem~\ref{thm: microlocal Garding} yields
\begin{equation*}
\tau^{-1/2}\| \Lsct^{s}\Opt(\tilde\alpha)\Opt(\chi_1)\rconj\|_+ + \|\Lsct^{-N}\rconj\|_+ \gtrsim  \tau^{-1/2} 
\| \Lsct^{s+1}\Opt(\chi_1)\rconj\|_+ ,
\end{equation*}
for every $N>0$.
From this and as $z_2-z_1=2i\Lsct^{-1}\Opt(\tilde\alpha)\Opt(\chi_1)\rconj$ from \eqref{def: z1 z2}, we deduce, using 
symbolic calculus and $\chi_2\chi_1=\chi_1$, 
\begin{align} 
	\label{eq: estimation rconj}
\tau^{-1/2} \| \Lsct^{s+1}\Opt(\chi_1)\rconj\|_+& \lesssim   \tau^{-1/2}\| \Lsct^{s+1}\Opt (\chi_2) \Lsct^{-1}\Opt(\tilde\alpha)\Opt(\chi_1)\rconj\|_+ + \|\Lsct^{-N}\rconj\|_+ \\
& \lesssim   \| \Lsct^{s+1/2}z_1\|_+ +\| \Lsct^{s+1/2}z_2\|_+ + \|\Lsct^{-N}\rconj\|_+, \notag
\end{align}
as $\tau^{-1/2}  \Lsct^{s+1}\Opt (\chi_2) $ is an operator of order $s+1/2$.

From the first equation of \eqref{eq: reslov bd 3} and  \eqref{hyp: mu eq tau},
we have 
\begin{align}  
\tau^{1/2}\| \Lsct^{s}\Opt (\chi_1) \tilde v\|_+
&\lesssim \| \Lsct^{s+1}\mu^{-1}\tau^{1/2}  \Opt (\chi_1) \rconj\|_++ \| \Lsct^{s} \tilde F\|_+ 
+ \| \Lsct^{s}   \rconj\|_+ \\
&\lesssim  \| \Lsct^{s+1/2}z_1\|_+ +\| \Lsct^{s+1/2}z_2\|_+ + \| \Lsct^{s}   \rconj\|_++ \| \Lsct^{s} \tilde F\|_+,  \notag
\end{align}
from \eqref{eq: estimation rconj}.

From \eqref{def: z1 z2}, $(z_2)_{|x_d=0}=    i\Lsct^{-1}\Opt(\tilde\alpha)\Opt(\chi_1)\rconj_{|x_d=0} $, arguing as from 
above and using the G\aa rding estimate of Theorem \ref{ASC8},
we have 
\begin{equation}
	\label{eq: est trace microloc conj Bis}
| \Lsct^{s+1/2} \Opt(\chi_1) \rconj_{|x_d=0}  |\lesssim | \Lsct^{s+1/2}  (z_2)_{|x_d=0} |.
\end{equation}
 From \eqref{eq: mult. 4Bis}--\eqref{eq: est trace microloc conj Bis}  we obtain Lemma \ref{lem: Carleman microloc}. 
  
 Now we consider the case $q-\mu^2=0$. Let $\eps>0$, we can shrink $W$ such that 
$|q-\mu^2|\le \eps \langle \xi',\tau\rangle^2$ in $W$. Note that $|\mu|\sim\tau\sim|\xi'|$ on $W$.
Let  $\chi_1$ be the cutoff defined previously supported on $W$ and 
$\chi_0$ supported on $W$ and $\chi_0=1$ on the support of $\chi_1$.
By symbolic calculus we have  
\begin{align}
	\label{eq: commutateur q chi}
\Opt(\chi_1)\Opt (\delta)\Opt (\zeta')&=\Opt (\delta)\Opt (\zeta')\Opt(\chi_1)+[ \Opt(\chi_1),\Opt (\delta)\Opt (\zeta')]
\\
&=\Opt(q)\Opt(\chi_1)+\Opt(r_1)\Opt(\chi_1)  \notag
\\
&\quad + \Opt(\chi_0)[ \Opt(\chi_1),\Opt (\delta)\Opt (\zeta')]+ \Opt(r_{-N}),\notag
\end{align}
where $r_1\in\Ssct^1$ and $r_{-N}\in\Ssct^{-N}$.
Observe that $\mu^{-1}\chi_j\in\Ssct^{-1}$ for $j=1,2$.

From \eqref{eq: 2x2}, \eqref{eq: commutateur q chi}, and
by symbolic calculus  we have
\begin{equation*}
	\left\{\begin{array}{ll}
	D_{x_d}\Opt(\chi_1)\rconj +i\tau(\pd_{x_d}\varphi)\Opt(\chi_1)\rconj -\mu \Opt(\chi_1)v^d=H_1 & \text{ in } x_d>0, \\
	D_{x_d}\Opt(\chi_1)v^d  + \mu^{-1} \Opt (q-\mu^{2})\Opt(\chi_1) \rconj
	+i\tau(\pd_{x_d}\varphi) \Opt(\chi_1)v^d=H_2 & \text{ in } x_d>0,\\   
	v^d=0 & \text{ on } x_d=0,
	\end{array}\right.
\end{equation*}
where
\begin{equation}
\label{eq: 2x2 bis}
	\| \Lsct^{s}  H_j\|_+  \le C_\eps
	\big(
 \|  \Lsct^{s}  F^d \|_+
+ \|  \Lsct^{s}  G\|_+   
    +\|  \Lsct^{s}    \rconj  \|_+
 + \| \Lsct^{s}   v^d  \|_+
+\| \Lsct^{s}  \tilde F  \|_+\big),
\end{equation}
for $ j=1,2$ with $C_\eps$ depends on $\eps$.
We compute
\begin{multline}
	\label{eq: alpha zero 1}
2\Re ( H_1,i\Lsct^{2s+1}\Opt(\chi_1)\rconj )_+\\
=2\Re (D_{x_d}\Opt(\chi_1)\rconj 
+i\tau(\pd_{x_d}\varphi)\Opt(\chi_1)\rconj -\mu \Opt(\chi_1)v^d,i\Lsct^{2s+1}\Opt(\chi_1)\rconj)_+.
\end{multline}
By microlocal G\aa rding inequality of Theorem~\ref{thm: microlocal Garding} we have, using  $\tau \sim |\xi'|$ on $W$
\begin{equation}
2\Re (i\tau(\pd_{x_d}\varphi)\Opt(\chi_1)\rconj,i\Lsct^{2s+1}\Opt(\chi_1)\rconj)_+\ge C_0\|  \Lsct^{s+1} \Opt(\chi_1) \rconj \|_+^2
-C_N\| \Lsct^{-N}\rconj \|^2
\end{equation}
for $C_0>0$, for all $N>0$ and $C_N>0$. 

\medskip

From this, \eqref{for: integ parts} and \eqref{eq: alpha zero 1} we obtain
\begin{multline}
2\Re ( H_1,i\Lsct^{2s+1}\Opt(\chi_1)\rconj )_+\\
\ge |\Lsct^{s+1/2}\Opt(\chi_1)\rconj_{|x_d=0}|^2 + C_1 \| \Lsct^{s+1}\Opt(\chi_1)\rconj \|_+^2-  \mu^2C_2 \| \Lsct^{s}\Opt(\chi_1)v^d \|_+^2
-C_N\| \Lsct^{-N}\rconj \|^2,
\end{multline}
for $C_1,\,C_2>0$, for all $N>0$ and $C_N>0$.

\medskip

We then obtain
\begin{multline}
		\label{est. square root 1}
 |\Lsct^{s+1/2}\Opt(\chi_1)\rconj_{|x_d=0}|^2 + \| \Lsct^{s+1}\Opt(\chi_1)\rconj \|_+^2
  \le  \mu^2 C_3  \| \Lsct^{s}\Opt(\chi_1)v^d \|_+^2
\\
+ C_\eps
	\big(
 \|  \Lsct^{s}  F^d \|_+
+ \|  \Lsct^{s}  G\|_+  
    + \|  \Lsct^{s}    \rconj  \|_+
 + \| \Lsct^{s}   v^d  \|_+
+\| \Lsct^{s}  \tilde F  \|_+\big),
\end{multline}
for $C_3>0$, for all $N>0$ and $C_N,C_\eps>0$.

\medskip

Now we compute
\begin{multline}
	\label{est. square root}
2\Re (H_2,i\mu\Lsct^{2s}\Opt(\chi_1)v^d)_+=2\Re(D_{x_d}\Opt(\chi_1)v^d, i\mu \Lsct^{2s}\Opt(\chi_1)v^d)_+\\
+2\Re(\mu^{-1} \Opt (q)\Opt(\chi_1) \rconj   -\mu \Opt(\chi_1)\rconj
	+i\tau(\pd_{x_d}\varphi) \Opt(\chi_1)v^d, i\mu \Lsct^{2s}\Opt(\chi_1)v^d)_+.
\end{multline}
From \eqref{for: integ parts} we have $2\Re(D_{x_d}\Opt(\chi_1)v^d  , i\Lsct^{2s}\Opt(\chi_1)v^d)_+=0$ as $v^d=0$ on $x_d=0$.

As $C\eps^2  \left\langle \xi',\tau\right\rangle^2 -\mu^{-2}(q-\mu^2)^2\ge \eps^2  \left\langle \xi',\tau\right\rangle^2 $, on $W$ with $C>0$, using that  $\tau\sim|\mu| \sim | \xi'|$, we have 
by microlocal G\aa rding inequality of Theorem~\ref{thm: microlocal Garding}  
\begin{multline*}
2\Re ( \mu^{-1} \Opt (q -\mu^2) \Opt ( \chi_1) \rconj   , \Lsct^{2s} \mu^{-1} \Opt (q -\mu ^2) \Opt (\chi_1) \rconj)_+
\\
\le  C_4 \eps^2 \|\Lsct^{s+1} \Opt(\chi_1)\rconj  \|_+^2 + C_{N,\eps}\|\Lsct^{-N} \rconj \|_+^2 .  
\end{multline*}
Then we have
\begin{multline}
	\label{eq: est. square root 2}
2\Re ( \mu^{-1} \Opt (q)\Opt(\chi_1) \rconj   -\mu \Opt(\chi_1)\rconj, i\mu \Lsct^{2s}\Opt(\chi_1)v^d)_+ \\
\le \eps |\mu| C_5 \|  \Lsct^{s}\Opt(\chi_1)v^d\|_+
\big( 
 \| \Lsct^{s+1}\Opt(\chi_1)\rconj  \|_++C_{N,\eps}\|\Lsct^{-N} \rconj \|_+
\big),
\end{multline}
for $C_5>0$.

From microlocal G\aa rding inequality of Theorem~\ref{thm: microlocal Garding} and as   $\pd_{x_d} \varphi(x_0)>0$, we have
\begin{equation}
\label{eq: square root bis}
2\Re(i\tau(\pd_{x_d}\varphi) \Opt(\chi_1)v^d, i\mu\Lsct^{2s}\Opt(\chi_1)v^d)_+\ge \mu^2C_6\|\Lsct^{s}\Opt(\chi_1) v^d  \|_+^2
-C_N\| \Lsct^{-N}v^d \|_+^2,
\end{equation}
where $C_6>0$ is independent of $\eps$, for all $N>0$, $C_N>0$.

\medskip

From \eqref{eq: 2x2 bis} and \eqref{est. square root}--\eqref{eq: square root bis} we obtain 
\begin{align*}
\mu^2 \|\Lsct^{s}\Opt(\chi_1) v^d  \|_+^2&\le  |\mu| \eps C_2\|  \Lsct^{s}\Opt(\chi_1)v^d\|_+ 
 \big( 
 \| \Lsct^{s+1}\Opt(\chi_1)\rconj  \|_++ C_{N,\eps}\|\Lsct^{-N} \rconj \|_+
\big)
 \notag \\
&\quad + C_\eps
	\big(
 \|  \Lsct^{s}  F^d \|_+
+ \|  \Lsct^{s}  G\|_+  + \|  \Lsct^{s}    \rconj  \|_+  + \| \Lsct^{s}   v^d  \|_+
 +\| \Lsct^{s}  \tilde F  \|_+\big)^2.
\end{align*}
We deduce 
\begin{align}
	\label{est square root 2}
\mu^2 \|\Lsct^{s}\Opt(\chi_1) v^d  \|_+^2&\le   \eps^2 C_7   
 \big( 
 \| \Lsct^{s+1}\Opt(\chi_1)\rconj  \|_+^2  \big)
 \notag \\
&\quad + C_\eps
	\big(
 \|  \Lsct^{s}  F^d \|_+
+ \|  \Lsct^{s}  G\|_+  + \|  \Lsct^{s}    \rconj  \|_+  + \| \Lsct^{s}   v^d  \|_+
+\| \Lsct^{s}  \tilde F  \|_+\big)^2.
\end{align}
By the linear  combination \eqref{est square root 2}$+\eps$\eqref{est. square root 1}  and fixing $\eps$ 
sufficiently small, from \eqref{hyp: mu eq tau}  and  $\tau$ sufficiently large, we deduce
\begin{multline}
	\label{eq: alpha = 0}
\tau\|\Lsct^{s}\Opt(\chi_1) v^d  \|_+
+ |\Lsct^{s+1/2}\Opt(\chi_1)\rconj_{|x_d=0}| +  \| \Lsct^{s+1}\Opt(\chi_1)\rconj \|_+\\
\lesssim 
 \|  \Lsct^{s}  F^d \|_+
+ \|  \Lsct^{s}  G\|_+   
    + \|  \Lsct^{s}    \rconj  \|_+
 + \| \Lsct^{s}   v^d  \|_+
+\| \Lsct^{s}  \tilde F  \|_+   . 
\end{multline}
From the first equation of \eqref{eq: reslov bd 3}  and from \eqref{hyp: mu eq tau} we have 
\begin{align}
	\label{eq: alpha = 0Bis}
\tau \| \Lsct^{s}\Opt (\chi_1) \tilde v\|_+
&\lesssim \| \Lsct^{s+1} \Opt (\chi_1) \rconj\|_++ \| \Lsct^{s} \tilde F\|_+ 
+ \| \Lsct^{s}   \rconj\|_+ \\
&\lesssim 
 \|  \Lsct^{s}  F^d \|_+
+ \|  \Lsct^{s}  G\|_+   
+\| \Lsct^{s}  \tilde F  \|_+ 
    + \|  \Lsct^{s}    \rconj  \|_+
 + \| \Lsct^{s}   v^d  \|_+
  , \notag
\end{align}
from \eqref{eq: alpha = 0}.
From \eqref{eq: alpha = 0} and \eqref{eq: alpha = 0Bis} we obtain Lemma \ref{lem: Carleman microloc} in the case $\alpha=0$.
\section{Logarithmic stability} \label{section4}
\setcounter{equation}{0}

The exponential estimate of Proposition~\ref{prop: resolv exp} is the consequence of the two following results. First a global 
Carleman estimate with an observability term and second an estimate of the observability term coming from the dissipation.

Let $\omega_0 $ and $\omega_1$ be open sets such that $\omega_1\Subset\omega_0\Subset \omega$, and,  from \eqref{eq:a},  we have $b(x)\ge b_{-} >0$ for $x\in\omega$. In what follows we denote by $\| .\|_0:=\| .\|_{L^2(\Omega)}$.
 %
 %
\begin{thm}
	\label{th: global carleman}
	Let $\Omega$ be an open bounded set of $\mathbb{R}^{d}$ with smooth boundary. Let $\varphi\in\mathscr{C}(\R^d)$
	 be a function 
	that satisfies the  sub-ellipticity assumption in $\overline{\Omega}\setminus \omega_1$. Then there exist 
	$\tau_{*}>0$ and $C>0$ such that
\begin{equation*}
\tau^{3/2}\|e^{\tau\varphi} r\|_0 +\tau^{3/2}\|e^{\tau\varphi} u\|_0 \leq C\left(\tau\|e^{\tau\varphi}g\|_{0}+\tau\|e^{\tau\varphi}f\|_{0}  
 +   \tau^{3/2}\|e^{\tau\varphi} r\|_{L^2(\omega_0)} +\tau^{3/2}\|e^{\tau\varphi} u\|_{L^2(\omega_0)}  \right),
\end{equation*}
for all $u, r \in\mathscr{C}_{c}^{\infty}(\overline{\Omega})$ which 
satisfies \eqref{eq: reslov bd}, $u \cdot n|_{\Gamma} = 0 $, $\tau\geq\tau_{*}$,  and $\mu$ 
satisfying \eqref{hyp: mu eq tau}.
\end{thm}
\begin{rem}   
It is classical that there exist $\psi$ such that $\partial_n \psi(x)<0$ for $x\in\partial \Omega$ and $\nabla \psi\ne0$ for 
$x\in \overline{\Omega}\setminus \omega_1$ (see Fursikov-Imanuvilov~\cite{FI:96}).  From Lemma~\ref{-11},  $\varphi=e^{\lambda\psi}$ satisfies 
sub-ellipticity condition 
in $ \overline{\Omega}\setminus \omega_1$ for $\lambda$ sufficiently large. In what follows we fix such a function $\varphi$.
\end{rem}
 %
 %
\begin{prop}
	\label{prop: dissipation est.}
Let $(u,r)\in {\mathcal D}({\mathcal A})$ solution of $({\mathcal A}_d +i\mu )(u,r)=(f,g)\in H$.
Then we have
\begin{align}
	\label{est: dissipation est.}
&|\mu|  \| \sqrt{b}u\|_0^2\le C\| ( u,r)\|_H\|(f,g)\|_H  \\
& |\mu|  \| r\|_{L^2(\omega_0)}^2\le C\| ( u,r)\|_H   \|(f,g)\|_H ,  \notag
\end{align}
for some constant $C>0$.
\end{prop}
From these two results we are able to prove Proposition~\ref{prop: resolv exp}.
%
%
\begin{proof}[Proof of Proposition~\ref{prop: resolv exp}] 
Noting that the resolvent problem $(\mathcal{A}-i\mu)(u,r)=(f,g)$ is written as follow
\begin{equation*}
\left\{\begin{array}{ll}
\nabla r+i\mu u=f-bu&\text{in }\Omega
\\
\div(u)+i\mu r=g&\text{in }\Omega
\\
u.n=0&\text{on }\Gamma.
\end{array}\right.
\end{equation*}
This allows us to apply Theorem~\ref{th: global carleman}. So let $C_2= \max_{x\in\overline{\Omega}} \varphi(x)$ and $C_1=\min_{x\in\overline{\Omega}} \varphi(x)$ we deduce from the Carleman
 estimate of Theorem~\ref{th: global carleman} that 
 \begin{equation}
\| r\|_0 +\| u\|_0 \lesssim e^{(C_2-C_1 )\tau}\left(\|g\|_{0}+\| f+bu\|_{0}  
 +  \|r\|_{L^2(\omega_0)} +\|u\|_{L^2(\omega_0)}  \right).
 \end{equation}
 Taking $\tau=|\mu|/c_0$ accordingly with \eqref{hyp: mu eq tau},  by the estimates of Proposition~\ref{prop: dissipation est.} and
as  $\| bu\|_{0} \lesssim  \| \sqrt{b}u\|_{0} $, 
 we have
  \begin{equation}
\| ( u,r)\|_H \lesssim Ce^{K|\mu|}\left(\|(f,g)\|_H  
 +\| ( u,r)\|_H^{1/2}\|(f,g)\|_H  ^{1/2} \right),
  \end{equation}
which yields $\| ( u,r)\|_H \lesssim e^{K'|\mu|}\|(f,g)\|_H  $. This is the sought result.
\end{proof}

%
%
\begin{proof}[Proof of Proposition~\ref{prop: dissipation est.}]
 From equation, we have $-\nabla r+ i\mu u-bu=f$ taking the inner product with $u$, we obtain
 $ - (\nabla r,u)  +i\mu\|u \|^2 - (bu,u) =(f,u)$. Integrating by parts, we have $- (\nabla r,u) =(r,\div u)$ as 
 $u\cdot n=0$ on $\partial\Omega$.  Using the second equation $-\div u+i\mu r=g $, we have $- (\nabla r,u) =(r,i\mu r -g)$.
We thus obtain
\begin{equation*}
 - i\mu\|r\|^2   - (r,g) +   i\mu\|u \|^2- (bu,u) =(f,u).
\end{equation*}
Taking the real part of this equation we have $|\mu| \|\sqrt{b}u\|^2\le |(f,u)| +|  (r,g) |$. This implies the first estimate 
of~\eqref{est: dissipation est.}.

Let $\chi\in \mathscr{C}_{c}^{\infty}(\R^d)$ such that $\chi(x)=1$ for $x\in\omega_0$ and $\chi$ supported in $\omega$.
Taking the inner product between   $-\div u+i\mu r=g $ and $\chi^2r$, we obtain $ (-\div u,\chi^2r)+i\mu \| \chi r\|^2=(g,\chi^2r)$.
Integrating by parts we have $ (-\div u,\chi^2r)= ( u,\chi^2\nabla r) + ( u,2\chi   r\nabla\chi)$ and by equation 
$-\nabla r+ i\mu u-bu=f$ we have
\begin{equation*}  
 - i\mu\| \chi u \|  ^2    -    ( u, \chi^2 bu )  - ( u,\chi^2f)       + ( u,2\chi   r\nabla\chi)+i\mu \| \chi r\|^2=(g,\chi^2r).
\end{equation*}
Taking account that $b\ge b_- $ in $\omega$, thus on the support of $\chi$, we have
\begin{equation*}  
|\mu| \| \chi r\|^2\lesssim \| ( u,r)\|_H\|(f,g)\|_H  + \|  u \nabla\chi\|  \| \chi r\|+|\mu |\| \chi u \|  ^2+  \| \sqrt{b}u\|^2 .
\end{equation*}
We can estimate $\|  u \nabla\chi\|$ and $ \| \chi u \| $ by $  \| \sqrt{b}u\|$ and by the first estimate of 
Proposition~\ref{prop: dissipation est.} we obtain the second estimate.
\end{proof}

%
%
\begin{proof}[Proof of Theorem~\ref{th: global carleman}] 
Let $x_0\in \overline{\Omega}\setminus \omega_1$, from Corollary~\ref{cor: Carleman interior H-1} if $x_0\in\Omega$ or 
from Theorem~\ref{th: Carleman bord} if $x_0\in\partial\Omega$ we obtain, in both cases,  an open neighborhood  
(in $\R^d$) of 
$x_0$, $V$   such that
\begin{equation}
	\label{est: Carleman local to global}
\tau^{3/2}\|e^{\tau\varphi} r\|_0 +\tau^{3/2}\|e^{\tau\varphi} u\|_0 \leq C\left(\tau\|e^{\tau\varphi}g\|_{0}+\tau\|e^{\tau\varphi}f\|_{0}  
  \right),
\end{equation}
for $u,  r\in\mathscr{C}_{c}^{\infty}(V)$. By compactness of $ \overline{\Omega}\setminus \omega_1$ we can find a finite recovering
$(V_j)_{j\in J}$ of $ \overline{\Omega}\setminus \omega_1$. Let $(\chi_j )_{j\in J}$ be a partition of unity subordinated 
to $(V_j)_{j\in J}$ such that $\sum_{j\in J}\chi_j(x)=1$ for $x\in  \overline{\Omega}\setminus \omega_1$. 
Let $u_j=\chi_j u$ and $r_j=\chi_j r$ where $(u,r)$ solution to \eqref{eq: reslov bd}, $u \cdot n|_{\Gamma} = 0 $.
We have 
\begin{align*}
&    -\nabla r_j+i\mu u_j=\chi_j f-r\nabla \chi_j \\
&   -\div u_j +i\mu r_j=\chi_j g-u\cdot \nabla \chi_j.
\end{align*}
We can apply the  Carleman estimate~\eqref{est: Carleman local to global} in each $V_j$ and we obtain
\begin{align*}
\tau^{3/2}\|e^{\tau\varphi} r_j\|_0 +\tau^{3/2}\|e^{\tau\varphi} u_j\|_0
&\lesssim \tau\|e^{\tau\varphi}(\chi_j g-u\cdot \nabla \chi_j)\|_{0}+\tau\|e^{\tau\varphi}(\chi_j f-r\nabla \chi_j)\|_{0}  \\
&\lesssim \tau\|e^{\tau\varphi}g\|_{0}
+\tau\|e^{\tau\varphi} u\|_{0}
+\tau\|e^{\tau\varphi}f\|_{0} 
+\tau\|e^{\tau\varphi}r\|_{0} .
\end{align*}
We have 
\begin{align*}
&\tau^{3/2}\|e^{\tau\varphi} r\|_0 +\tau^{3/2}\|e^{\tau\varphi} u\|_0  \\
&\qquad  \lesssim \tau^{3/2}\sum_{j\in J}\big( \|e^{\tau\varphi} r_j\|_0 +\|e^{\tau\varphi} u_j\|_0 \big)
+ \tau^{3/2}\|e^{\tau\varphi} r\|_{L^2(\omega_0)} +\tau^{3/2}\|e^{\tau\varphi} u\|_{L^2(\omega_0)}   \\
&\qquad  \lesssim   \tau\|e^{\tau\varphi}g\|_{0}
+\tau\|e^{\tau\varphi} u\|_{0}
+\tau\|e^{\tau\varphi}f\|_{0} 
+\tau\|e^{\tau\varphi}r\|_{0}
+ \tau^{3/2}\|e^{\tau\varphi} r\|_{L^2(\omega_0)} +\tau^{3/2}\|e^{\tau\varphi} u\|_{L^2(\omega_0)} .
\end{align*}
This gives the sought result as we can absorb the term $\tau\|e^{\tau\varphi} u\|_{0}+\tau\|e^{\tau\varphi}r\|_{0} $
with the left hand side.
\end{proof}


\end{document}